\documentclass[10pt]{amsart}

\pdfoutput=1
\usepackage{geometry} % see geometry.pdf on how to lay out the page. Shere's lots.
\usepackage{graphicx} 
\usepackage{amsmath, amsfonts, amssymb, amsthm}
\usepackage{advdate}
\usepackage{hyperref}
\usepackage[utf8x]{inputenc}
\usepackage{changes}
\usepackage{subcaption}
\usepackage{enumerate}
\usepackage{adjustbox}

\DeclareRobustCommand{\bbone}{\text{\usefont{U}{bbold}{m}{n}1}}

\usepackage{bm}
\newtheorem{theorem}{Theorem}
\theoremstyle{remark}
\newtheorem{remark}{Remark}
\newtheorem{lemma}[theorem]{Lemma}

\newcommand{\tr}{\intercal}
\newcommand{\p}{{\mathsf{p}}}
\newcommand{\m}{{\mathsf{m}}}
\newcommand{\q}{{\mathsf{q}}}
\newcommand{\g}{{\mathsf{g}}}

\newcommand{\de}{\mathrm{d\,}}
\newcommand{\lla}{\left\langle}
\newcommand{\rra}{\right\rangle}

\newcommand{\at}[2][]{#1|_{#2}}

\usepackage{tikz}

\hypersetup{
			colorlinks=true,
			linkcolor=black,
                        linktoc=page,
			anchorcolor=black,
			citecolor=blue,
			urlcolor=black,
 }

\title{Adaptive Elastic-Net estimation for sparse diffusion processes}
\author{Alessandro De Gregorio}
\address{Department of Statistical Sciences, ``Sapienza" University of Rome,
	P.le Aldo Moro, 5 - 00185, Rome, Italy}
\email{alessandro.degregorio@uniroma1.it}

\author{Dario Frisardi}
\address{Department of Statistical Sciences, ``Sapienza" University of Rome,
	P.le Aldo Moro, 5 - 00185, Rome, Italy}
\email{dario.frisardi@uniroma1.it}

\author{Francesco Iafrate}
\address{Department of Mathematics, University of Hamburg,
Bundestr. 55, 20146 Hamburg, Germany}
\email{francesco.iafrate@uni-hamburg.de}
\author{Stefano Iacus}
\address{Institute for Quantitative Social Science, Harvard University, 1737 Cambridge Street
CGIS Knafel Building, Room K350
Cambridge, MA 02138, USA}
\email{siacus@iq.harvard.edu}
\date{\today}

\begin{document}

\begin{abstract}

Penalized estimation methods for diffusion processes and dependent data have recently gained significant attention due to their effectiveness in handling high-dimensional stochastic systems. In this work, we introduce an adaptive Elastic-Net estimator for ergodic diffusion processes observed under high-frequency sampling schemes. Our method combines the least squares approximation of the quasi-likelihood with adaptive $\ell_1$ and $\ell_2$ regularization. This approach allows to  enhance prediction accuracy and interpretability while effectively recovering the sparse underlying structure of the model.

In the spirit of analyzing high-dimensional scenarios, we provide finite-sample guarantees for the (block-diagonal) estimator's performance by deriving high-probability non-asymptotic bounds for the $\ell_2$ estimation error. These results complement the established oracle properties in the high-frequency asymptotic regime with mixed convergence rates, ensuring consistent selection of the relevant interactions and achieving optimal rates of convergence. Furthermore, we utilize our results to analyze one-step-ahead predictions, offering non-asymptotic control over the $\ell_1$ prediction error.

The performance of our method is evaluated through simulations and real data applications, demonstrating its effectiveness, particularly in scenarios with strongly correlated variables.
\end{abstract}

\maketitle
\tableofcontents
\section{Introduction}

Nowadays, the regularization methods are very useful  for a comprehensive understanding of the underlying parametric model which generally is supposed sparse; that is some coefficients are exactly zero.  The main idea is to perform simultaneously the selection of the true model and to estimate the parameters. Penalized estimators of a parameter $\theta$ are generally defined as follows
\begin{equation}\label{eq:penmeth}\hat\theta_n \in \arg\min_{\theta\in\overline\Theta}\{\mathfrak L_n(\theta)+p(\theta)\}\end{equation}
where $\mathfrak L_n(\theta)
$ is a contrast function (negative log-likelihood or sum of squared residuals) based on a random sample of size $n$, $\Theta$ is the parameter space and $p(\theta)$ is the penalty function. If $p(\theta)=\lambda |\theta|^q, \lambda>0,$ where $q\in(0,1],$ we have the Bridge estimator introduced in \cite{frank1993statistical}, which for $q=1$ reduces to LASSO (Least Absolute Shrinkage Selection Operator) studied in \cite{tibshirani1996regression} for  linear regression. Furthermore, a good selection procedure should have the so-called oracle properties as discussed in \cite{fan2001variable, fan2004nonconcave, fan2006statistical}.

Penalized estimation for stochastic processes is a quite recent research topic in the field of statistical learning for random complex system and discrete-time dependent data, and in particular,  shrinkage estimators has been applied to multidimensional ergodic diffusion processes
\begin{equation}    \label{eq:sde1}
    \mathrm{d}X_t=b(X_t,\alpha)\mathrm{d}t+\sigma(X_t,\beta)\mathrm{d}W_t, \quad X_0=x_0,\end{equation} where $x_0$ is a deterministic initial point, $b: \mathbb{R}^d\times\Theta_\alpha\to\mathbb{R}^d$ and $\sigma: \mathbb{R}^d\times\Theta_\beta\to\mathbb{R}^d\otimes\mathbb{R}^r$ are Borel known functions (up to $\alpha$ and $\beta$) and $(W_t)_{t\geq0}$ is a $r$-dimensional standard Brownian motion.

In the high frequency setting, regularized estimation problems for discretely observed  low-dimensional sparse stochastic differential equations    \eqref{eq:sde1} (i.e., the dimensions  of the unknown parameters $\alpha$ and $\beta$ are fixed)
 have been dealt with in \cite{de2012adaptive}, \cite{masuda2017moment}, \cite{suzuki2020penalized}, \cite{kinoshita2019penalized} and \cite{de2021regularized}. The authors used penalized selection procedures \eqref{eq:penmeth} based on LASSO ($\ell_1$ constraints) and Bridge type penalties ($\ell_2$ constraints). The asymptotic oracle properties of the regularized estimators are derived:  selection consistency (i.e., consistently estimates null parameters as zero) and optimal rate asymptotic normality of the true subset model.  In \cite{10.1214/18-EJS1436}, the LASSO estimator has been analyzed for diffusion processes with small noise observed at continuous time. 

Recently, some papers dealt with penalized methods  for  high-dimensional diffusion processes where the number of parameters as well as the dimension of the model is large; for instance the reader can consult \cite{gaiffas2019sparse}, \cite{fujimori2019dantzig}, \cite{10.1214/20-EJS1775}, \cite{ciolek2022lasso}, \cite{10.3150/22-BEJ1574} and \cite{amorino2024sampling}.  Usually in this framework the statistical analysis focuses on regularized estimators for the drift term and some non-asymptotic oracle bounds are derived.

The Elastic-Net procedure (see \cite{zou2005regularization} and \cite{zou2009adaptive}) is a regularization method for a linear regression model developed to enhance the performance of both lasso and ridge regression techniques. It uses a combination of $\ell_1$ and $\ell_2$ penalties in the regularization term; namely the estimator \eqref{eq:penmeth} involves $ p(\theta)=\lambda_1|\theta|+\lambda_2|\theta|^2.$ The Elastic-Net allows to handle multicollinearity, more effectively than lasso regression, with the capability of selection of groups of correlated predictors together and then improving the prediction of the underlying model (usually LASSO procedures select one predictor of the group). For an insightful discussion of this method  see  \cite{zou2005regularization}.

In this paper, we introduce an adaptive Elastic-Net estimator for ergodic diffusion processes observed under high-frequency sampling schemes. Up to our knowledge, this is the first attempt to study the Elastic-Net problem for the statistical analysis of stochastic differential equations. Our method combines the least squares approximation introduced in \cite{wang1} with the adaptive $\ell_1$ and $\ell_2$
regularization; namely, let $\theta=(\alpha,\beta),$ the objective function appearing in \eqref{eq:penmeth} becomes
$$\mathfrak L_n(\theta)+ L_n(\theta)+R_n(\theta),$$
where $ L_n(\theta)$ is the adaptive LASSO penalty, while $ R_n(\theta)$
 is the Ridge term. The idea is to replace the quasi-log-likelihood function of \eqref{eq:sde1} with its second order Taylor expansion and adding the adaptive $\ell_1$ and $\ell_2$ penalties.  It is worth to mention that a crucial issue is the choice of the initial non-penalized estimator (i.e. the point where the approximation takes place), which have to satisfy some standard properties.   By means of this approach the estimator obtained from minimization problem \eqref{eq:penmeth} enhances prediction accuracy and interpretability while effectively recovering
the sparse underlying structure of the model. In this way it is possible to overcome some limitations of the LASSO regularization, especially when we consider multivariate diffusion process having correlated groups of variables. A part of our analysis is devoted to prove some asymptotic properties of the Elastic-Net estimator . Indeed, it is well-known that a good selection procedure should have some requirements; in particular, it should satisfy the so-called oracle properties (see, e.g., \cite{fan2001variable} and \cite{fan2006statistical}).

Furthermore, we introduce a block-diagonal Elastic-Net estimator, which is asymptotic equivalent to the original estimator, and prove that some non-asymptotic bounds for the estimation error (under suitable assumptions on the contrast function). These last results are particularly interesting because allow to provide deeper insight into the high-dimensional features of the estimator, by obtaining bounds (with high probability) depending on the dimension of the parametric space. It is worth mentioning that our non-asymptotic results is a first step toward the analysis of high-dimensional scenarios in this setting. We are finally able to  obtain finite sample guarantees for the mean absolute error of predictions at a future time. Our theoretical findings  are backed by empirical results that show that introducing a flexible Elastic-net penalty leads to better model selection and more accurate forecasting, when compared to LASSO regularization. 

The paper is organized as follows. In Section \ref{sec:defen}, we introduce the sparse parametric diffusion model and the underlying asymptotic regime. Furthermore, by means of the least squares approximations approach, the adaptive Elastic-Net estimator for stochastic differential equations observed at discrete times is defined. In Section \ref{sec:tp} the theoretical features of the estimator are discussed. In particular we get the oracle properties of the estimator, which are crucial in order to define a reasonable selection procedure. Non-asymptotic bounds, involving under suitable assumptions the parametric dimension, are introduced in Section \ref{sec:nab} for a block-diagonal Elastic-Net estimator. This latter is asymptotically equivalent to the estimator defined in Section \ref{sec:defen}. Furthermore, the bounds represents a first insight into the high-dimensional theory  for the problem addressed in the paper. Section \ref{sec:pe} is devoted to the analysis of the prediction error; in particular the mean absolute error is studied and its non-asymptotic bounds discussed. In Section \ref{sec:sim}-\ref{sec:real}, we assess the performance of the adaptive Elastic-Net estimator, by means a numerical analysis on simulated data and an application of our methodology to well-being real data (comparing Italy and Japan during the Covid outbreak). The optimization path algorithm is briefly described.  All the proofs are collected in the last section.

\section{Definition of Elastic-Net estimator for diffusions sampled at discrete times}\label{sec:defen}

Let  $(\Omega, \mathcal{F}, \mathbf{F}=(\mathcal{F}_t)_{t\geq0}, P)$ be a filtered complete probability space. Let us consider a $d$-dimensional solution process $X:=(X_t)_{t\geq0}$ to the following stochastic differential equation (SDE) \begin{equation}\label{eq:sde}
    \mathrm{d}X_t=b(X_t,\alpha)\mathrm{d}t+\sigma(X_t,\beta)\mathrm{d}W_t, \quad X_0=x_0,\end{equation} where $x_0$ is a deterministic initial point, $b: \mathbb{R}^d\times\Theta_\alpha\to\mathbb{R}^d$ and $\sigma: \mathbb{R}^d\times\Theta_\beta\to\mathbb{R}^d\otimes\mathbb{R}^r$ are Borel known functions (up to $\alpha$ and $\beta$) and $(W_t)_{t\geq0}$ is a $r$-dimensional standard ${\bf F}$-Brownian motion.
    
     We assume that $\alpha\in\Theta_\alpha\subset\mathbb{R}^{\p}$, $\beta\in\Theta_\beta\subset\mathbb{R}^{\q}$, $\p,\q\in\mathbb{N}$, are unknown parameters and $\Theta_\alpha$, $\Theta_\beta$ are compact convex sets.
     In this framework, the parameter of interest is  $\theta:=(\alpha, \beta)^\tr$ and the parametric space is given by $\Theta:=\Theta_\alpha\times\Theta_\beta\subset\mathbb{R}^\m$, where $\m:=\p+\q$.

     The true value of $\theta$ is indicated by $\theta_0:=(\alpha_0,\beta_0)^\tr\in \mathbb{R}^\m$. Furthermore $\theta_0\in\mathrm{Int}(\Theta)$ and $0\in\mathbb{R}^\m$ belongs to $\Theta$. $P:=P_{\theta_0}$ stands for the probability law of $X.$ 

   We assume  that   the stochastic differential equation $X$ represents a sparse parametric model. The sparsity condition of the parametric model is introduced by imposing that some coefficients in $\theta_0$ are exactly zero; i.e. $\p^0:=\left|\{j:\alpha_{0,j}\neq 0\}\right|$,  $\q^0=\left|\{h:\beta_{0,h}\neq 0\}\right|$ and $\m^0:=\p^0+\q^0$.

 Let us recall some notations. Let $\langle A, x^{\otimes 2} \rangle=$ tr$(Ax^{\otimes 2})=x^\tr A x,$ where $x^{\otimes 2}= xx^\tr,$ $x\in\mathbb R^\m$ $A\in \mathbb R^\m\otimes \mathbb R^\m.$ Furthermore the euclidean norm of the vector $x$ is indicated by $|x|$ and $||A||=\sqrt{\tau_{\max}(A A^\tr)}.$ We sometimes adopt the notation $[n]$ to denote the set ${1, 2, \ldots , n}$.

     Hereafter, we suppose that $X$ satisfies the following  regularity conditions.
     \begin{itemize}
\item (Existence and uniqueness) There exists a constant $C$ such that
$$\sup_{\alpha\in\Theta_\alpha}|b(x,\alpha)-b(y,\alpha)|+\sup_{\beta\in\Theta_\beta}||\sigma(x,\beta)-\sigma(y,\beta)|| \leq C|x-y|,\quad x,y\in \mathbb R^d.$$
\item (Non-degeneracy) $\inf_{x,\beta}$ det $\Sigma(x,\beta)>0,$ where $\Sigma(x,\beta):=\sigma\sigma^\tr(x,\beta).$
\item (Smoothness) $b$ and $\sigma$ are smooth functions; i.e. they are continuously differentiable (in both variables) up to some order and their derivatives are uniformly polynomial growth.  
\item (Ergodicity)  $X$ is ergodic; i.e. there exists a unique invariant probability measure $\mu=\mu_{\theta_0}$ such that for any bounded measurable function $g:\mathbb{R}^d\to\mathbb{R}$ $$\frac{1}{T}\int_0^{T} g(X_t)\mathrm{d}t\overset{p}{\longrightarrow}\int_{\mathbb{R}^d}g(x)\mathrm{d}\mu,\quad T\longrightarrow\infty.$$
\item (Moments) $\sup_t E[|X_t|^k]<\infty$ for all $k>0.$
\item (Identifiability) $b(x,\alpha)=b(x,\alpha_0)$ for $\mu$-a.e. $x$ if and only if $\alpha=\alpha_0$ and det($\Sigma(x,\beta))=$ det($\Sigma(x,\beta_0)$) for $\mu$-a.e. $x$ if and only if $\beta=\beta_0.$
     \end{itemize}

\begin{remark}
 The ergodic property representing a kind of law of large numbers is required for the estimation of the parameter $\alpha$ appeared in the drift term. 
 The ergodicity of the process $X$ follows from some mixing conditions. For instance, we assume that there exists a positive constant $a$ such that
 $$\nu_X(u)\leq \frac{e^{-au}}{a},\quad u>0$$
 where
 $$\nu_X(u)=\sup_{t\geq 0}\sup_{\underset{B\in\sigma\{X_r:r\geq t+u\}}{A\in\sigma\{X_r:r\leq t\}}}|P(A\cap B)-P(A)P( B)|.$$
Other conditions are the following ones: if we assume that $\Sigma$ is bounded and there exist positive constants $\lambda_{-},\lambda_{+}$ and $\Lambda$ such that for all $\beta$
$$0<\lambda_{-}\leq \langle \Sigma(x,\beta)x/|x|, x/|x|\rangle\leq \lambda_{+},\quad \frac{\text{Tr}(\Sigma(x,\beta))}{d}\leq \Lambda$$
and for all $\alpha$
$$\langle b(x,\alpha), x/|x|\rangle\leq - r|x|^a,\quad |x|\geq M_0, $$
with $M_0\geq 0, a\geq -1$ and $r>0.$ Under this assumptions the process $X$ is ergodic and the moments condition holds (see \cite{pardoux2001poisson}).
 
 The smoothness conditions for $b$ and $\sigma$ depend on the adopted estimation procedure. For this reason, we avoid to explicit this assumption and refer to the papers where the estimators are analyzed.     
\end{remark}

     The sample path of $X$ is observed at $n + 1$ equidistant discrete times $t_i^n$, such that $t_i^n-t_{i-1}^n=\Delta_n<\infty$ for $i=1,\ldots,n$ with $t_0^n=0$. 
    The collected data are the discrete observations of the sample path of $X$, that we represent by $\mathbf{X}_n:=(X_{t_i^n})_{0\leq i\leq n}$. The main goal of the paper is to estimate $\theta$ and simultaneously select the true subset model from the sample $\mathbf{X}_n$.

     %Let $e$ be an integer with $e\geq 2$, 
     
     The asymptotic scheme adopted in this paper is the following $$n\Delta_n\longrightarrow\infty, n\Delta^2\longrightarrow 0 \text{ \ and \ } \Delta_n\longrightarrow 0$$ as $n\to \infty.$ We have high frequency sampling and long-run data. Actually, it is possible to weak the above asymptotic condition imposing $n\Delta_n^k\longrightarrow 0, k> 2.$ In this setting, some estimators for ergodic diffusions have been proposed in \cite{kess} and \cite{uchida2012adaptive}.
     %and there exists $\epsilon\in\left(0, \frac{e-1}{e}\right)$ such that $$n^\epsilon\leq n\Delta_n$$ for large $n$.

In order to introduce the Elastic-Net estimation methodology for diffusion processes, we take into account the same approach developed in \cite{wang1},  \cite{de2012adaptive}, \cite{suzuki2020penalized} and \cite{de2021regularized}. Let us start by dealing with  a generic contrast function $\mathfrak{L}_n(\theta):=\mathfrak{L}_n(\theta; \mathbf{X}_n)$ such that the estimator  $\tilde{\theta}_n:=\tilde{\theta}_n(\mathbf{X}_n)$ of $\theta$  is given by: 
\begin{align}\label{eq:iniest}
    \tilde{\theta}_n:=(\tilde{\alpha}_n, \tilde{\beta}_n)^\tr \in \arg\min_\theta \mathfrak{L}_n(\theta)
\end{align} 
If $\mathfrak{L}_n(\theta)$ is twice differentiable with respect to $\theta$, it is approximated by Taylor expansion as follows: \begin{align*}
       \mathfrak{L}_n(\theta) & \simeq \mathfrak{L}_n(\tilde{\theta}_n) +\frac{1}{2}\langle \ddot{\mathfrak{L}}_n(\tilde{\theta}_n), (\theta-\tilde{\theta}_n)^{\otimes 2}\rangle
    \end{align*} 
    where $\ddot{\mathfrak{L}}_n$ represents the Hessian matrix of the second derivatives with respect to $\theta$. Therefore, we can minimize $\langle \ddot{\mathfrak{L}}_n(\tilde{\theta}_n), (\theta-\tilde{\theta}_n)^{\otimes 2}\rangle$ instead of $\mathfrak{L}_n(\theta).$ This remark inspires the following definition.

Usually, $\tilde\theta_n$ minimizes a loss function; nevertheless we will not assume the existence of a loss function.     Let $\hat{G}_n$ be a $\m\times\m$ almost surely positive definite symmetric random matrix depending on $n$. We introduce the following objective function: 
     \begin{align} \label{objfunction}
       \mathcal{F}_n(\theta;\tilde{\theta}_n)&=\langle \hat{G}_n, (\theta-\tilde{\theta}_n)^{\otimes 2}\rangle+L_n(\theta)+R_n(\theta)\notag\\
       &=| \hat{G}_n^{1/2} (\theta-\tilde{\theta}_n)|^2+L_n(\theta)+R_n(\theta)
    \end{align} where $ \hat{G}_n^{1/2}$ is the symmmetric matrix such that $\hat{G}_n^{1/2}\hat{G}_n^{1/2}=\hat{G}_n$ and $L_n(\theta)$ and $R_n(\theta)$ are function of the parameter of interest and represent the penalization term of Elastic-Net in the objective function. In particular, $L_n(\theta)$ represents the LASSO component of the penalization term:
    \begin{equation}\label{eq:lasso-pen}
        L_n(\theta)=|\alpha|_{1,\kappa_n}+ |\beta|_{1,\pi_n}
    \end{equation} involving the weighted $\ell_1$ norms $$|\alpha|_{1,\kappa_n}=\sum_{j=1}^{\p}\kappa_{n,j}|\alpha_j|$$ and $$|\beta|_{1,\pi_n}=\sum_{h=1}^{\q} \pi_{n,h}|\beta_h|,$$ where $\kappa_{n,j}, j=1,...,\p,$ and $\pi_{n,h},h=1,...,\q,$ are sequence of real positive random variable depending on the data and representing the amount of shrinkage for $\alpha_j$ and $\beta_h$, respectively. 
    \begin{remark}
            Usually, the adaptive weights are chosen as follows (\cite{zou2006adaptive})
    \begin{align}
\kappa_{n,j}&=\frac{\lambda_{1,n}}{|\tilde\alpha_{n,j}|^{\delta_1}},\quad j=1,\ldots,\p, \label{eq:adap-w1}\\ 
\pi_{n,h}&=\frac{\gamma_{1,n}}{|\tilde\beta_{n,h}|^{\delta_2}},\quad h=1,\ldots,\q,\label{eq:adap-w2}
    \end{align}
    where $\lambda_{1,n}$ and $\gamma_{1,n}$ are suitable sequences of positive numbers and $\delta_1,\delta_2>0$. The idea is that $\tilde\theta_n$ is consistent and then this choice allows to assign  higher weights to the zero coefficients. 
    \end{remark}
 %Since $\tilde\theta_n$ is usually consistent, in order to avoid the blowing up of the adaptive weights, we could set
   % \begin{align}
%\kappa_{n,j}&=\frac{\lambda_{1,n}}{(|\tilde\alpha_{n,j}|+\frac1n)^{\delta_1}},\quad %j=1,\ldots,\p,\\
%\pi_{n,h}&=\frac{\gamma_{1,n}}{(|\tilde\beta_{n,h}|+\frac{1}{n\Delta_n})^{\delta_2}},\quad h=1,\ldots,\q.
 %   \end{align}
   % {\bf (Si potrebbe prendere l'enet non adattivo come fa Zou?)}
    
    The Ridge component of the Elastic-Net penalty is:
    \begin{equation}\label{eq:ridge-pen}
    R_n(\theta)=\lambda_{2,n}|\alpha|^2+\gamma_{2,n}|\beta|^2
    \end{equation} where $\lambda_{2,n}$ and $\gamma_{2,n}$ are sequence of positive real values, depending on $n$.

We observe that if $\tilde\theta_n$ coincides with \eqref{eq:iniest}, by setting $\hat G_n=\ddot{\mathfrak{L}}_n(\tilde{\theta}_n),$ the first term in $\mathcal{F}_n$ is exactly the least squares approximation  of the loss function $\mathfrak{L}_n(\theta).$

     The Elastic-Net estimator $\hat{\theta}_n:=\hat{\theta}_n(\mathbf{X}_n):\mathbb{R}^{(n+1)\times d}\longrightarrow \overline{\Theta}$ is defined minimizing the  objective function \eqref{objfunction}; i.e. \begin{equation}  \label{eq:estenet}
        \hat{\theta}_n:=(\hat{\alpha}_n,\hat{\beta}_n)^\tr\in \arg\min_{\theta\in\overline{\Theta}}\mathcal{F}_n(\theta;\tilde{\theta}_n).
    \end{equation}

The primary advantage of our approach lies in addressing a penalized optimization problem that involves a convex cost function. This characteristic enhances the feasibility of the estimation procedure both theoretically and computationally.

\begin{remark}\label{rem:enetlasso}
The adaptive LASSO estimator $\hat\theta_n$(LASSO) for stochastic differential equations studied in \cite{de2012adaptive} is defined as follows
$$\hat\theta_n\text{(LASSO)}\in\arg\min_{\theta}\left\{| \hat{G}_n^{1/2} (\theta-\tilde{\theta}_n)|^2+L_n(\theta)\right\}$$
or equivalently
\begin{equation}\label{eq:lasso}
\hat\theta_n\text{(LASSO)}\in\arg\min_{\theta}\left\{\langle \hat G_n,\theta^{\otimes 2}\rangle-2\tilde\theta_n^\tr\hat G_n\theta +L_n(\theta)\right\}.
\end{equation}
Clearly, the stationary points   of $\mathcal{F}_n(\theta;\tilde{\theta}_n)$ coincides with those of the function
$$\langle \hat G_n+C(\lambda_{2,n},\gamma_{2,n}),\theta^{\otimes 2}\rangle-2\tilde\theta_n^\tr\hat G_n\theta +L_n(\theta),$$
where $C(\lambda_{2,n},\gamma_{2,n}):=\text{diag}(\lambda_{2,n} {\bf I}_{\p},\gamma_{2,n}{\bf I}_{\q}).$ Therefore the Elastic-Net estimator becomes
\begin{equation*}
\hat\theta_n\in\arg\min_{\theta}\left\{\langle \hat G_n+C(\lambda_{2,n},\gamma_{2,n}),\theta^{\otimes 2}\rangle-2\tilde\theta_n^\tr\hat G_n\theta +L_n(\theta)\right\};
\end{equation*}
i.e. $\hat\theta_n$ is a stabilized version of \eqref{eq:lasso} where the Ridge adjustment is taken into account  by the perturbation of $ \hat G_n$ with the rates $\lambda_{2,n}$ and $\gamma_{2,n}$. 
\end{remark}

\section{Theoretical properties of Elastic-Net estimator}\label{sec:tp}

Without loss of generality, we set $\alpha_{0,j}\neq0$, $j=1,\ldots,\p^0,$ and $\beta_{0,h}\neq 0$, $h=1,\ldots,\q^0$.  Let $A_n:=\mathrm{diag}\left(\frac{1}{\sqrt{n\Delta_n}}\mathbf{I}_{\p}, \frac{1}{\sqrt{n}}\mathbf{I}_{\q}\right),$ where $\mathbf{I}_{\g}$ stands for the identity matrix of size $\g$.

We introduce the following assumptions.

\begin{itemize}
    \item[P1.] The initial estimator $\tilde{\theta}_n:=\left(\tilde{\alpha}_n,\tilde{\beta}_n  \right)^\tr\,:\mathbb{R}^{(n+1)\times d}\longrightarrow \overline{\Theta}$ of $\theta$ is consistent in a mixed-rates asymptotic regime, $\tilde{\alpha}_n$ is $\sqrt{n\Delta_n}$-consistent while $\tilde{\beta}_n$ is $\sqrt{n}$-consistent: $$\left(\sqrt{n\Delta_n}(\tilde{\alpha}_n-\alpha_0), \sqrt{n}(\tilde{\beta}_n-\beta_0)\right)^\tr = O_p(1)$$
    
    % Let $\hat{D}_n:=A_n\hat{G}_n A_n$. There exists a $\m\times\m$ positive definite symmetric random matrix $G$, such that: $$\hat{D}_n\overset{p}{\longrightarrow}G$$
    
    \item[P2.] $\tilde{\theta}_n$ is asymptotically normal; i.e. $$\left((\sqrt{n\Delta_n}(\tilde{\alpha}_n-\alpha_0), \, \sqrt{n}(\tilde{\beta}_n-\beta_0)\right)^\tr \overset{d}{\longrightarrow} N_\m\left(0,\Gamma(\theta_0)^{-1}\right),$$ where $\Gamma(\theta_0):=\mathrm{diag}\left(\Gamma^{^{\alpha\alpha}}(\theta_0), \Gamma{^{\beta\beta}}(\beta_0)\right)$ is the asymptotic Fisher information and \begin{align*}
        \Gamma^{^{\alpha\alpha}}(\theta_0) & := \left[\int_{\mathbb{R}^d} \partial_{\alpha_i} b^\tr(x,\alpha_0) \Sigma^{-1}(x,\beta_0) \partial_{\alpha_j} b(x, \alpha_0)\mu(\mathrm{d}x)\right]_{i,j=1}^\p\\
        \Gamma^{^{\beta\beta}}( \beta_0) & := \left[\frac{1}{2}\int_{\mathbb{R}^d} \mathrm{tr}\left[(\partial_{\beta_i} \Sigma) \Sigma^{-1}(\partial_{\beta_j}\Sigma)\Sigma^{-1}(x, \beta_0)\right]\mu(\mathrm{d}x)\right]_{i,j=1}^\q.
    \end{align*}% and $\partial_\alpha := \left(\frac{\partial}{\partial_{\alpha_1}},\ldots, \frac{\partial}{\partial_{\alpha_{\p}}}\right)^\tr$ and $\partial_\beta := \left(\frac{\partial}{\partial_{\beta_1}},\ldots, \frac{\partial}{\partial_{\beta_{\q}}}\right)^\tr$. %The integrability and the non-degeneracy of $\Gamma^{^{\alpha\alpha}}(\theta_0)$ and $\Gamma^{^{\beta\beta}}(\beta_0)$ are assumed.
    % The estimator $\tilde{\theta}_n$ is consistent: $$A_n^{-1}(\tilde{\theta}-\theta_0)=\left(\frac{1}{r_{n}}(\tilde{\alpha}_n-\alpha_0), \frac{1}{s_{n}}(\tilde{\beta}_n-\beta_0)\right)^\tr=O_p(1)$$ 
    
    \item[P3.] Let $G$ be a $\m\times\m$ matrix assumed to be symmetric and positive definite and such that: $$\hat{D}_n:=A_n\hat{G}_nA_n \overset{p}{\longrightarrow}G.$$
    % The estimator $\tilde{\theta}_n$ is asymptotically normal: $$A_n^{-1}(\tilde{\theta}-\theta_0)=\left(\frac{1}{r_n}(\tilde{\alpha}_n-\alpha_0), \frac{1}{s_n}(\tilde{\theta}_n-\theta_0)\right)^\tr \overset{d}{\longrightarrow} N_{\m}(0, J)$$ where $J=\Gamma^{-1}$ and $\Gamma$ is a $\m\times\m$ positive definite symmetric matrix.
    \item[P4.]  There exist two positive definite symmetric random matrices $G^{\alpha\alpha}\in\mathbb R^\p\otimes \mathbb R^\p$ and $G^{\beta\beta}\in\mathbb R^\q\otimes \mathbb R^\q$ such that $$\hat{D}_n\stackrel{p}{\longrightarrow}G:= \text{diag}(G^{\alpha\alpha},G^{\beta\beta}).$$
\end{itemize}

\begin{remark}
The conditions P1 and P2 reveal the mixed-rates asymptotic nature of a good estimation procedure for ergodic diffusion processes.  This is the main reason leading to introduce in the penalty terms $L_n$ and $R_n$ two different norms for each group of parameters. 
\end{remark}

\begin{remark}\label{rem:ini-ql}
The  maximum quasi-likelihood  estimator, indicated by $\tilde\theta_n^{\text{(QL)}},$ is obtained by maximizing the contrast function
given by the quasi-log-likelihood function \begin{align}\label{qlik}
\ell_n(\theta)
&:=-\frac12\sum_{i=1}^n\left\{\log\text{det}(\Sigma( X_{t_{i-1}^n},\beta))
+\frac{1}{\Delta_n}\langle \Sigma^{-1}( X_{t_{i-1}^n},\beta),(X_{t_i^n}-X_{t_{i-1}^n}-\Delta_n b( X_{t_{i-1}^n},\alpha))^{\otimes 2}\rangle\right\};
\end{align}
i.e. $\tilde\theta_n^{\text{(QL)}}\in\arg\max_\theta \ell_n(\theta).$
The terms appearing in the sum \eqref{qlik} 
represent a local-Gaussian approximation of the transition density of $X$ between two observations $X_{t_{i-1}^n}$ and $X_{t_{i}^n},$ arising from the Euler-Maruyama discretization scheme (see \cite{kloeden1992}). This approach was originally introduced in \cite{florens1989approximate} with $\sigma(x,\beta)=\beta$. Therefore we can assume that $\mathfrak L_n(\theta)=-\ell_n(\theta)$ and 
a possible choice of $\hat G_n$ is the Hessian matrix $-\ddot \ell_n(\tilde\theta_n^{\text{(QL)}})$ evaluated in $\tilde\theta_n^{\text{(QL)}}$. Under some mild conditions the estimator  $\tilde\theta_n^{\text{(QL)}}$ and $\ddot \ell_n(\tilde\theta_n^{\text{(QL)}})$ behave properly and  the properties P1, P2 and P3 are satisfied. In particular 
$$\hat{D}_n=A_n (-\ddot\ell_n(\tilde\theta_n^{\text{(QL)}})) A_n\overset{p}{\longrightarrow}\Gamma(\theta_0)$$
(actually uniformly in $\theta$).
The maximum quasi-likelihood estimation for the ergodic stochastic differential equations has been developed, for instance, in \cite{yoshida2011polynomial} by means a polynomial-type inequality for suitable statistical random fields.   

Other estimators for ergodic diffusions satisfying the asymptotic properties P1, P2 and P3 appeared in literature. For instance, the quasi-Bayesian estimator is studied in \cite{yoshida2011polynomial}, while the hybrid multistep estimator has been introduced in \cite{kamatani2015hybrid}. An approach based on the approximate martingale estimating functions is discussed in \cite{sorensen2024efficient}. If $n\Delta_n^k\longrightarrow 0, k\geq  2,$   estimators defined by local-Gaussian  contrast functions involving higher-order expansions of the conditional moments of $X_t$, have been proposed in \cite{kess} and \cite{uchida2012adaptive}.
\end{remark}

Let us introduce $a_{n}:=\max\{\kappa_{n,j};j\leq \p^0\}$, $b_{n}:=\min\{\kappa_{n,j};j> \p^0\}$, $c_{n}:=\max\{\pi_{n,h};h\leq \q^0\}$ and $d_{n}:=\min\{\pi_{n,h};h> \q^0\}$.    We introduce the following conditions.

\begin{itemize}
    \item[A1.] $\frac{a_{n}}{\sqrt{n\Delta_n}}=O_p(1)$ and $\frac{c_{n}}{\sqrt{n}}=O_p(1)$.
    \item[A2.] $\frac{\lambda_{2,n}}{\sqrt{n\Delta_n}} = O(1)$ and $\frac{\gamma_{2,n}}{\sqrt{n}} = O(1)$,% (e.g. $\lambda_n = \lambda$ and $ \gamma_n =\gamma$).
    \item[A3.] $\frac{a_n}{\sqrt{n\Delta_n}}=o_p(1)$ and $\frac{c_{n}}{\sqrt{n}}=o_p(1)$.
    \item[A4.] $\frac{b_n}{\sqrt{n\Delta_n}}\overset{p}{\longrightarrow}\infty$ and $\frac{d_n}{\sqrt{n}}\overset{p}{\longrightarrow}\infty.$\footnote{A sequence of random variables $X_n$ converges to $\infty$ in probability, we write $X_n\stackrel{p}{\to}\infty$ if $P(X_n > K)\to 1$ for every $K>0$}.
\end{itemize}

We denote $\alpha_\star:=(\alpha_1, \dots,\alpha_{\p^0})^\tr$, $\beta_\star:=(\beta_1, \dots,\beta_{\q^0})^\tr$ and  $\alpha_\bullet:=(\alpha_{\p^0+1}, \dots,\alpha_{\p})^\tr$, $\beta_\bullet:=(\beta_{\q^0+1}, \dots,\beta_{\q})^\tr$. For the discussion of the next result we need to introduce the following notation on a block matrix $M\in\mathbb R^u\otimes \mathbb R^v,$ where $u,v\in\{\p,\q\},$   
	 $$M =\left(  \begin{matrix} % or pmatrix or bmatrix or Bmatrix or ...
	      M_{\star\star} &   M_{\star\bullet} \\
	        M_{\bullet\star} &   M_{\bullet\bullet} \\
	   \end{matrix}	\right),$$  
	   where for $u^0,v^0\in\{\p^0,\q^0\},$ 
	   \begin{itemize}
	   \item  $M_{\star\star}=(m_{ij})_{1\leq i\leq u^0,1\leq j\leq v^0}$ is a $u^0\times v^0$ matrix;
	   \item $M_{\star\bullet}=(m_{ij})_{1\leq i\leq u^0,v^0< j\leq v},$ is a $u^0\times (v-v^0)$ matrix;
	   \item $M_{\bullet\star}=(m_{ij})_{u^0< i\leq u ,1\leq j\leq v^0}$ is a  $(u-u^0)\times v^0$ matrix;
	   \item  $M_{\bullet\bullet}=(m_{ij})_{u^0< i\leq u,v^0<j\leq v}$ is a  $(u-u^0)\times(v-v^0)$ matrix.
	   \end{itemize}
    Furthermore, under P4 we introduce the following $\m^0\times \m$ matrix
 	   $$\mathfrak G:= \left(  \begin{matrix} % or pmatrix or bmatrix or Bmatrix or ...
	          \mathfrak G_\alpha&0\\
	        0 &  \mathfrak G_\beta 
	      \end{matrix}\right),$$
    	      where $  \mathfrak G_\alpha:=({\bf I}_{\p^0}\,\,\, ( G_{\star\star}^{\alpha\alpha})^{-1} G^{\alpha\alpha}_{\star\bullet})$ and  $  \mathfrak G_\beta:=({\bf I}_{\q^0}\,\,\, ( G_{\star\star}^{\beta\beta})^{-1} G^{\beta\beta}_{\star\bullet}).$ 

Now, we are able to prove the crucial asymptotic Oracle properties for the estimator \eqref{eq:estenet}. 
\begin{theorem}[Oracle properties]\label{thm:asoracle} The adaptive Elastic-Net estimator $\hat\theta_n$ satisfying the following properties.
\begin{itemize}
    \item[i)] (Consistency) Let us assume that P1, P3, A1 and A2 hold. Then $$A_n^{-1}\left(\hat{\theta}_n-\theta_0\right)=O_p(1).$$
    % Assume A1, A2, B1, B2, then $$A_n^{-1}\left(\hat{\theta}_n-\theta_0\right)=O_p(1)$$
    
  \item[(ii)] (Selection consistency) Under the assumptions P1, P3, A1, A2 and A4, we have that:$$P(\hat{\alpha}_{n\bullet}=0)\longrightarrow1 \text{ \ and \  }P(\hat{\beta}_{n\bullet}=0)\longrightarrow1$$ as $n\longrightarrow\infty$. %Then $P(\hat{\theta}_{n\bullet}=0)\longrightarrow 1$.
    % If assumptions A1, A2, B1, B2 and B4 are satisfied, we have that $$P(\hat{\alpha}_{n\bullet}=0)\longrightarrow1 \text{ \ and \  }P(\hat{\beta}_{n\bullet}=0)\longrightarrow1$$ as $n\longrightarrow\infty$

\item[(iii)] (Asymptotic normality)
Let us assume that P1, P2, P4, A3 and A4 are fulfilled.  Then $$\left(\sqrt{n\Delta_n}(\hat{\alpha}_n-\alpha_0)_\star, \sqrt{n}(\hat{\beta}_n-\beta_0)_\star\right)^\tr\overset{d}{\longrightarrow}N_{\m^0}\left(0,\mathfrak G\Gamma(\theta_0)^{-1}\mathfrak G \right)$$ as $n\longrightarrow \infty$. Furthermore if $G=\Gamma(\theta_0)$
$$ \mathfrak G\Gamma(\theta_0)^{-1}\mathfrak G =\mathrm{diag}\left(\Gamma^{^{\alpha\alpha}}_{\star\star}(\theta_0)^{-1}, \Gamma^{^{\beta\beta}}_{\star\star}(\beta_0)^{-1} \right).$$
 \end{itemize}   
\end{theorem}

We recall that a sequence of random variables $(X_n)_{n\geq 1}$ is uniformly $L^r$-bounded if and only if $\sup_{n\geq1}\mathbb{E}\left[|X_n|^r\right]<\infty,$ for all $r\geq 1.$ % For simplicity, we assume $\lambda_{2,n}=\lambda_2$ and $\gamma_{2,n}=\gamma_2,$ for all $n\geq 1.$

\begin{theorem}[Uniform $L^r$-boundness]\label{thm:bound}
By assuming A2 and $(\hat{D}_{n})_{n\geq 1},$ $(\hat{D}_{n}^{-1})_{n\geq 1},$ $(A_n^{-1}(\tilde{\theta}_n-\theta_0))_{n\geq 1}, $ $(\frac{a_n}{\sqrt{n\Delta_n}})_{n\geq 1},$   $(\frac{c_n}{\sqrt n})_{n\geq 1}$uniformly $L^r$-bounded, then  

\begin{equation}
    \sup_{n\geq1}\mathbb{E}\left[|A_n^{-1}(\hat{\theta}_n-\theta_0)|^r\right]<\infty, \quad r\geq 1.    \end{equation} 
\end{theorem}

\begin{remark}\label{rem:weights}
  The assumptions required in Theorem \ref{thm:bound} are not  strong since, for instance, the maximum quasi-likelihood and Bayesian estimators are uniformly $L^r$-bounded (see  \cite{yoshida2011polynomial}). Furthermore, we can chose the weights properly, in order to get $L^r$-bounded coefficients $a_n$ and $c_n;$  for instance by setting
  $$\kappa_{n,j}=\frac{\lambda_{1}}{|\tilde\alpha_{n,j}|^{\delta_1}+a_n}, \quad \pi_{n,h}=\frac{\gamma_{1}}{|\tilde\beta_{n,h}|^{\delta_2}+b_n}$$
  or
  $$\kappa_{n,j}=\frac{\lambda_{1}}{(|\tilde\alpha_{n,j}|\vee a_n)^{\delta_1}}, \quad \pi_{n,h}=\frac{\gamma_{1}}{(|\tilde\beta_{n,h}|\vee b_n)^{\delta_2}}$$
  with $a_n,b_n>0$ and such that $a_n,b_n\downarrow 0$ and $\lambda_{1}, \gamma_{1}>0.$
  
  The $L^r$ estimates are useful, for example, in the in asymptotic decision theory,  where the estimator is efficient if it attains the Hajek's minimax bound, or for the validation of the information criterion.
\end{remark}

\section{Block-diagonal estimator and non-asymptotic bounds}\label{sec:nab}
Before proceeding with the results in this section, dealing with the non-asymptotic properties of the estimator, we state the following assumptions which we need in order to control the finite-sample behaviour of the contrast function and information matrix.
Denote by $\partial_\theta \ell_n$
and  $\partial^2_{\theta, \theta} \ell_n$ the gradient and Hessian matrix of $\ell_n$, respectively, and by $\partial_\theta \overline{\ell_n} = A_n \partial_\theta \ell_n$, 
$\partial^2_{\theta \theta} \overline{\ell_n} = A_n
\partial^2_{\theta \theta} \ell_n A_n$ their scaled version.

We need the following two conditions.

\begin{enumerate}[]
    \item[A5$(r)$.]
    (\emph{Regular contrast}) 
    The negative quasi-likelihood and its $\theta$-derivatives $\ell_n$,  $\partial_\theta \ell_n$,  $\partial_{\theta\theta}^2\ell_n$ can be extended continuously
    to the boundary of $\Theta$ and, for $r>0$, there exist square integrable random variables $\xi_n$ with $\mathbb E\xi_n^2 \leq J$ and $\mu > 0$ s.t. 
    \begin{align*}
      &(i)\quad  \max_{i \in [\p + \q]} \sup_{\theta:| \theta_0 - \theta| \leq r}|
        \partial_{\theta_i} \overline{\ell_n}(\theta) | \leq \xi_n,
\\
  &(ii) \quad \inf_{v\in \mathbb R^{\p+\q}:|v|=1} \inf_{\theta: |\theta - \tilde \theta_n| \leq  r}  v^\top \partial^2_{\theta\theta } \overline{\ell_n}(\theta) v \geq \mu
    \end{align*}
     for all $n$, $P_{\theta_0}$-a.s..

\end{enumerate}

\begin{enumerate}[]
    \item[A6.]
    (\emph{Regular information}) 
    The random matrix $\hat D_n$ satisfies
    \[
    \tau_1 \leq \tau_{\min} (\hat D_n)\leq  \tau_{\max} (\hat D_n) \leq \tau_2
    \]
    for some $\tau_1, \tau_2 >0$, for all $n,$  $P_{\theta_0}$-a.s..
\end{enumerate}

\begin{remark}
    Condition A5($r$)-(\emph{ii}) can be seen as  a finite sample version of the so-called identifiability condition (see, e.g., \cite{ciolek2022lasso},\cite{amorino2024sampling}).
\end{remark}

\begin{lemma}[Non-asymptotic bound for QMLE] \label{lem:non-asy}
Let $\tilde \theta_n^{(QL)} = (\tilde \alpha_n^{(QL)}, \tilde \beta_n^{(QL)})$ be  the QL estimator. Under A5$(r)$, on the event \{$|\tilde \theta_n^{(QL)} - \theta_0| \leq r$\}, we have 

\begin{equation}
    \sqrt{n \Delta_n} |\tilde\alpha_n^{(QL)} - \alpha_0| \leq \frac{2 \xi_n}{\mu} \sqrt{\p}, 
    \qquad 
    \sqrt{n} |\tilde\beta_n^{(QL)} - \beta_0| \leq \frac{2 \xi_n}{\mu} \sqrt{\q}.
\end{equation}

\end{lemma}

In this section, we derive non-asymptotic results for the Elastic-Net estimator in the case where the $\hat G_n$ information matrix is taken to be block-diagonal. This is a mild restriction, since in most cases we have that the off-diagonal blocks of the scaled information matrix $\hat D_n^{\alpha \beta}$ converge to zero in probability. This is the case in the notable example of the hessian matrix in quasi-likelihood estimation, when evaluated at $\tilde \theta_n^{(QL)}$.

Specifically, let us assume $\hat G_n=\text{diag}(\hat G_n^{\alpha\alpha}, \hat G_n^{\beta\beta})$ which satisfies condition P4.  In this case the objective function can be decomposed as follows
$$\mathcal{F}_n(\theta;\tilde{\theta}_n)=\mathcal{F}_{1,n}(\alpha;\tilde{\alpha}_n)+\mathcal{F}_{2,n}(\beta;\tilde{\beta}_n)$$
where
\begin{align*}
       \mathcal{F}_{1,n}(\alpha;\tilde{\alpha}_n)&=\langle \hat{G}_n, (\alpha-\tilde{\alpha}_n)^{\otimes 2}\rangle+
       |\alpha|_{1,\kappa_n}+\lambda_{2,n}|\alpha|^2\notag\\
       &=| (\hat{G}_n^{\alpha\alpha})^{1/2} (\alpha-\tilde{\alpha}_n)|^2+
       |\alpha|_{1,\kappa_n}+\lambda_{2,n}|\alpha|^2
    \end{align*}
    and
    \begin{align*} 
       \mathcal{F}_{2,n}(\beta;\tilde{\beta}_n)&=\langle \hat{G}_n^{\beta\beta}, (\beta-\tilde{\beta}_n)^{\otimes 2}\rangle+|\beta|_{1,\pi_n}+\gamma_{2,n}|\beta|^2\notag\\
       &=| (\hat{G}_n^{\beta\beta})^{1/2} (\beta-\tilde{\beta}_n)|^2+|\beta|_{1,\pi_n}+\gamma_{2,n}|\beta|^2
    \end{align*}
    Furthermore $\hat\alpha_n\in \arg\min_\alpha \mathcal{F}_{1,n}(\alpha;\tilde{\alpha}_n)$ and $\hat\beta_n\in \arg\min_\beta \mathcal{F}_{2,n}(\beta;\tilde{\beta}_n).$

The following theorem provides non-asymptotic bounds for the Elastic-Net estimator. 

\begin{theorem}[Non-asymptotic bounds]\label{thm:ineq} Let us assume $\hat G_n=\text{diag}(\hat G_n^{\alpha\alpha}, \hat G_n^{\beta\beta})$ satisfying P4.  The error bounds of the Elastic-Net estimators $\hat\alpha_n$ and $\hat\beta_n$ are given by
\begin{align}\label{eq:nabound1a}
        |\hat\alpha_n-\alpha_0|\leq \frac{2}{n\Delta_n\tau_{\min}(\hat D_n^{\alpha\alpha})+\lambda_2}\left(\lambda_2|\alpha_0|+ n\Delta_n\tau_{\max}(\hat D_n^{\alpha\alpha})  |\tilde\alpha_n-\alpha_0|+\lambda_1|\kappa_n|\right),
    \end{align}
    and
            \begin{align}\label{eq:nabound1b}
        |\hat\beta_n-\beta_0| \leq \frac{2}{n \tau_{\min}( \hat D_n^{\beta\beta})+\gamma_2}\left(\gamma_2|\beta_0|+  n\tau_{\max}(\hat D_n^{\beta\beta})|   \tilde\beta_n-\beta_0|+\gamma_1 |\pi_n|\right).
    \end{align}
    
\end{theorem}

\begin{remark}
    If $\frac{\lambda_{2,n}}{\sqrt{n\Delta_n}}\longrightarrow 0$ and $\frac{\gamma_{2,n}}{\sqrt{n}}\longrightarrow 0,$ one has immediately that  P1 leads to
    $$A_n(\hat G_n+C(\lambda_{2,n},\gamma_{2,n}))A_n\stackrel{p}{\longrightarrow} G.$$
    Therefore, Remark \ref{rem:enetlasso} implies that the Oracle properties showed in \cite{de2012adaptive} and \cite{de2021regularized} hold true for $\hat\theta_n.$
\end{remark}

In the previous theorem the properties of the estimator depend on the error of the initial estimator, which in particular depends implicitly on the dimension, with the advantage of requiring very little assumptions.
In the following theorem we gather a deeper insight into the high-dimensional features of the estimator, by obtaining a high-probability result. We need to introduce some extra assumptions in order to control the finite sample behavior of the contrast and the Hessian matrix. Moreover we state the following results in the case where the initial estimator is the quasi-likelihood estimator. 

\begin{theorem}[Non-asymptotic bounds II]\label{thm:ineq-ii} Suppose that $\tilde \theta_n =\tilde \theta_n^{(QL)}$. For any $r>0$, under the assumptions of Theorem \ref{thm:ineq}, and, in addition, under A5($r/n\Delta_n$) - A6, the Elastic-Net estimators $\hat\alpha_n$ and $\hat\beta_n$ satisfies the bounds 

\begin{align}\label{eq:nabound2a}
        |\hat\alpha_n-\alpha_0| \leq 
        \frac{2}{n\Delta_n \tau_1 +\lambda_2} 
        \left(\lambda_2|\alpha_0|+ 
        \frac{\tau_2 \xi_n}{\mu} \sqrt{n\Delta_n \, \p}
        +\lambda_1|\kappa_n|\right),
    \end{align}
    and
\begin{align}\label{eq:nabound2b}
        |\hat\beta_n-\beta_0| \leq \frac{2}{n \tau_1 + \gamma_2}
        \left(\gamma_2|\beta_0|+  
        \frac{\tau_2 \xi_n}{\mu}  \sqrt{n \, \q}
        +\gamma_1 |\pi_n|\right)
\end{align}
 with probability at least $1 - C_L/r^L$, for some $L>0$.
 \end{theorem}

\begin{remark}
    In the previous theorem the parameter $r>0$ plays the role of a tuning parameter for the above inequalities. In particular it controls the finite-sample regularity of the contrast in  a neighborhood of 
    the maximum likelihood estimator. For fixed $n$, the larger the desired probability, the larger has to be the neighborhood in which A5$(r/n\Delta_n)$ holds. At the same time as $n$ increases the regularity requirement is less restrictive. Note also that this property holds asymptotically, in view of the fact that, under mild regularity assumptions, the quasi-likelihood satisfies an asymptotically quadratic approximation (see   \cite{yoshida2011polynomial}).

    Furthermore, we can deal with a pure high-dimensional setting by assuming  $\p=\p_n=O((n\Delta_n)^{\nu_1}), \nu_1\in(0,1),$ and $\q=\q_n=O(n^{\nu_2}), \nu_2\in(0,1).$ In this case the parameters tend to be large as $n$ increases and by exploiting the bounds \eqref{eq:nabound2a} and \eqref{eq:nabound2b} (and condition A3), it is possible to conclude that $\hat\alpha_n$ and $\hat \beta_n$ are consistent.    
\end{remark}

%\section{Non-ergodic estimator}

\section{Prediction error}\label{sec:pe}
We now turn our attention to the analysis of the prediction error. Our aim is to obtain finite sample guarantees for the mean absolute error of predictions at a future time. 
Assume that
\begin{enumerate}
	\item[A7.] The drift function is Lipschitz continuous w.r.t. $\alpha$
	with constant possibly depending on $x$ and uniformly bounded in $L^2(P)$; i.e.
	\[
	|b(x, \alpha_1) - b(x, \alpha_2)| \leq C(x) |\alpha_1 - \alpha_2| \quad  \forall \alpha_1, \alpha_2 \in \Theta_\alpha
	\]
	and $D=\sup_t \mathbb E C^2(X_t) < \infty$.
\end{enumerate}

\begin{remark}
    Assumption A7 is for example satisfied by a linear drift,  or when the drift is of the form $b(x, \alpha) = \sum_j \alpha_j b_j(x)$
    for Lipschitz functions $b_j$. We cannot just assume that $C(x) $ is finite a.s., as it would fail even in the linear drift case. 
\end{remark}
%We exploit a numerical scheme and its strong order of convergence to 
%obtain a prediction error bound.

 The construction of the predictor is based on the Euler-Maruyama Euler-Maruyama approximation of $X$, starting from $X_{T}$,  that is 
\begin{equation}\label{eq:euler}
    \tilde X_{T + h} = \tilde X_{T + h}(W, \theta) = X_{T} + b(X_{T}, \alpha) h + \sigma(X_{T}, \beta)\Delta W_{h}
\end{equation}

where $\Delta W_{h}=W_{T + h} - W_{T} $.  By standard arguments it can be shown that
\[
\mathbb E[ |\tilde X_{T + h} -  X_{T + h}|^2|\mathcal F_T]  \leq A h
\]
where the constant $A$ depends on the Lipschitz constant of $b$ and $\sigma$ and the moments of $X_0$, 
but not on $h$ (see \cite{kloeden1992}, Th 10.2.2) 
Our goal is to obtain a prediction for the future value $X_{T+h}$, but we cannot directly exploit \eqref{eq:euler} since in general we do not have access to the underlying Brownian motion $W$ nor to the true parameter $\theta_0$. For this reason we introduce an independent Brownian motion $W'$ and, given a $\mathcal F_T$-measurable estimator $\hat \theta$, we define the one-step predictor
\begin{equation}\label{eq:onestep-pred}
    \hat{X}_{T+h} = \mathbb{E}[
    \tilde X_{T + h}(W', \hat \theta) | \mathcal F_T
    ] = 
    X_T +  b(X_{T}, \hat \alpha)h
\end{equation}

Given a sample $\{X_{t_i}, i \in [n]\}$
and the Elastic-Net estimator satisfying the 
assumptions of Theorem 3, we give an error bound for the prediction error associated with the estimator $\hat{X}_{T_n+h}$, where $T_n = n\Delta_n$ is fixed.

\begin{theorem}[Non-Asymptotic prediction error bounds] \label{theo:prederr} Under the assumptions of \autoref{thm:ineq}, A6 and A7 we have that
    \begin{align}\label{eq:pred-bound}
        \mathsf{MAE}(\hat X_{T_n + h}) &= \mathbb E|X_{T_n + h} - \hat X_{T_n + h}| \notag
        \\
        & \leq 
        \sqrt h C_1 + 
        h C_2 +
        \frac{ C_3 h}{T_n\tau_1+\lambda_2}
        \left(\lambda_2 |\alpha_0|+  \tau_2 T_n  \mathbb E |\tilde\alpha_n-\alpha_0| +\lambda_1 \mathbb E|\kappa_n|\right)
    \end{align}
\end{theorem}

The intuition behind the previous result is as follow. Suppose $h$ is small, so that $\sqrt h + h \lesssim \sqrt h$, and that $n$ is large, so that, heuristically, we can write 
$ \mathbb E T_n |\tilde\alpha_n-\alpha_0|^2 \sim \p $. 
Considering adaptive weights as in Remark \ref{rem:weights}, by taking $a_n \sim \sqrt {T_n}$, $\delta_1 = 1$, one has that $E|\kappa_n| \lesssim \sqrt {T_n}$.
Equation \eqref{eq:pred-bound} essentially tells us that the prediction mean absolute error behaves  as 
\begin{equation}\label{eq:pred-bound-intuition}
\mathsf{MAE}(\hat X_{T_n + h}) \lesssim \sqrt h + 
 \frac{h \sqrt \p}{\sqrt{T_n}} + O(T_n^{-1/2}).
\end{equation}
The first term plays the role of an \emph{irreducible} error related to the forecasting task, showing that the predictions degrade with the forecasting horizon. The second source of error is related to the estimation task: it improves with larger sample sizes, it worsens in higher dimensions.

\section{Elastic-Net Optimization and Numerical Analysis}\label{sec:sim}

\subsection{Optimization Algorithm and Coefficient Path}
To obtain an explicit solution in an optimization problem where the objective function \eqref{objfunction} can be decomposed as $\mathcal{F}_n(\theta;\tilde{\theta}_n)=g+h$, with $g=\langle \hat{G}_n, (\theta-\tilde{\theta}_n)^{\otimes 2}\rangle+R_n(\theta)$ being convex and differentiable, and $h=L_n(\theta)$ being convex but not differentiable, iterative proximal gradient algorithms can be used. Specifically, accelerated gradient methods use a weighted combination of the current estimate and the previous gradient directions, and the proximal map function is iteratively applied till convergence of the algorithm. 

In particular, in the objective function \eqref{objfunction} we can consider the adaptive $\ell_1$-penalty as the function $h$ in the above decomposition. It is possible to use the coordinate descent iterative algorithm where, at iteration $t$, the update is: \begin{equation}
    \hat{\theta}^{t+1}_k = \arg\min_{\theta_k} \mathcal{F}_n\left(\hat{\theta}_1^t, \ldots,\hat{\theta}_{k-1}^t,\theta_k,\hat{\theta}_{k+1}^t,\ldots,\hat{\theta}^t_{\p+\q}\right)
\end{equation}

Fixing the sample size $n$, the coefficients in \eqref{eq:adap-w1}, \eqref{eq:adap-w2} and \eqref{eq:ridge-pen} no longer depend on $n$. Specifically, we consider the penalization constant $\lambda\cdot\gamma$ for the LASSO component (instead of using $\lambda_{1,n}$ and $\gamma_{1,n}$), while, similarly, we consider $\lambda(1-\gamma)$ for the Ridge component (instead of $\lambda_{2,n}$ and $\gamma_{2,n}$), where $\lambda>0$ and $\gamma \in (0,1].$ The tuning term $\gamma$ controls the relative proportion between the regularization penalties $\ell_1$ and $\ell_2$; i.e. $\gamma=1$, it reduces to LASSO.

Let $w_k$ be the adaptive weight associated with the parameter $\theta_k$, which represents an element of the diagonal matrix $W=\mathrm{diag}(\kappa_{n,1},\ldots, \kappa_{n,\p}, \pi_{n,1},\ldots,\pi_{n,\q})$. The coordinate descent updates take the form:
\begin{equation}\label{eq:cd-enet}
	\hat \theta ^t_k  = \frac{1}{1 + \frac{\lambda}{\hat G_{(k,k)}} (1 - \gamma)} \mathcal{S}_{\frac{\lambda \gamma}{{\hat G_{(k,k)}}} w_k} \left(\tilde \theta_k - \frac{1}{\hat G_{(k,k)} } \sum_{i \neq k} \hat{G}_{(k,i)} (\hat \theta^{t-1}_i - \tilde \theta_i) \right)
\end{equation} or equivalently
\begin{equation}
\hat \theta ^t_k = \frac{1}{\hat G_{(k,k)} + {\lambda} (1 - \gamma)} 
\mathcal{S}_{\lambda \gamma w_k} 
\left(\hat G_{(k,k)} \hat \theta_k^{t-1} -  \hat{G}_{k \cdot} 
(\hat \theta^{t-1}- \tilde \theta) \right) 
\end{equation} for $k=1,\ldots,\p+\q$, $t\geq 1$, where $\hat{G}_{(i,k)}$ is the element in the $i$-th row and $k$-th column of $\hat{G}_n$ and $\mathcal{S}_\mu(z):=\mathrm{sign}(z)(z-\mu)_+$ is the soft-thresholding operator. 

From this expression it is immediate to check that the largest $\lambda$ value that is necessary to consider, $\lambda_{\max}$, is:
\[
\inf\{ \lambda > 0: \hat \theta (\lambda) = 0 \} \leq \| (\gamma W)^{-1} \hat G \tilde \theta\|_{\infty} =: \lambda_{\max}
\]

A more efficient algorithm for the computation of the solution path can be obtained by applying Proximal Accelerated Gradient Descent algorithms. In the LASSO case a popular algorithm of this type is the so-called FISTA(\cite{beck2009fast}). See \cite{hastie2015statistical} for an introduction to the subject. In order to obtain the solution of
\eqref{objfunction}, we apply Algorithm 1 in \cite{degregorio2024pathwiseoptimizationbridgetypeestimators},
which is based on the proximal map for the Elastic-Net problem, given by
\begin{align}
	\mathrm{prox}^{EN}_{s, \lambda, w} &= 
	\arg \min_u \left\{ \frac 12 (x-u)^2 + \lambda \gamma w |u| + \frac{\lambda (1 - \gamma)}{2} u^2\right\}
	\\\notag
	&=
	\frac{1}{1 + \lambda s (1 - \gamma)}\mathrm{sign}(x) (|x| - \lambda \gamma s w)_+
	\\\notag
	&=
	\frac{1}{1 + \lambda s (1 - \gamma)} \mathcal{S}_{\lambda \gamma s w} (x)
\end{align}
where $s$ denoted the stepsize. 

Improvements beyond the reported algorithm can be achieved by leveraging  block-wise Proximal Alternating Minimization (referred to in the literature as PALM, see \cite{bolte2014proximal}), which updates one block of parameters at each step (see \cite{degregorio2024pathwiseoptimizationbridgetypeestimators}, for further insight regarding pathwise solution algorithms).

In a diffusion process, as in \eqref{eq:sde}, the classical cross-validation technique for tuning parameter selection $\lambda$ is not applicable due to the data dependency structure. For this reason, a data-driven technique is used, based on an iterative algorithm that considers the Euler discretization of the solution of \eqref{eq:sde} and, given estimates of the parameters, a score function of the current estimated residuals (see \cite{de2021regularized}). 
Another criterion for choosing $\lambda$ is to use the AIC method, penalizing the log-likelihood function \eqref{qlik}.

Figure \ref{fig:enet-path} shows the solution path computed by the coordinate descent updates \eqref{eq:cd-enet}
for an Ornstein-Uhlenbek model
$$
\de X_t = -B X_t \de t + A \de W_t
$$
where $A, B$ are parameter SPD matrices in $\mathbb R^{d\times d}$, with $d=5$.

\begin{figure}[ht!]
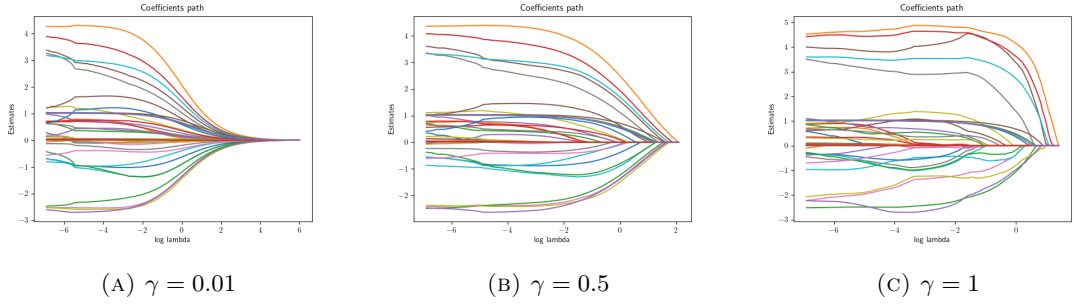

	\centering
	\begin{subfigure}[b]{0.32\textwidth}
		\includegraphics[width=\textwidth]{enet001.pdf}
		\caption{$\gamma = 0.01$}
		\label{fig:sub1}
	\end{subfigure}
	\hfill % This adds a horizontal space between the images
	\begin{subfigure}[b]{0.32\textwidth}
		\includegraphics[width=\textwidth]{enet05.pdf}
		\caption{$\gamma = 0.5$}
		\label{fig:sub2}
	\end{subfigure}
	\hfill % This adds a horizontal space between the images
	\begin{subfigure}[b]{0.32\textwidth}
		\includegraphics[width=\textwidth]{enet1.pdf}
		\caption{$\gamma = 1$}
		\label{fig:sub3}
	\end{subfigure}
	\caption{Elastic-Net coefficient paths as the tuning parameter $\lambda$ for different values of the mixing parameter $\gamma$.}
	\label{fig:enet-path}
\end{figure}

Among all the $\lambda$ values that make up the solution path for the Elastic-Net, the optimal value, $\lambda_{opt}$, is the largest $\lambda$ in the sequence such that the loss is within a median range of the minimum,
thereby ensuring a more stable selection. 
These algorithms are implemented in the Python library \texttt{sdelearn} (see \cite{sdelearn}).

\subsection{Stochastic regression model}

In this subsection we test our findings by applying our method to synthetic data. We evaluate the model selection and predictive capabilities of our Adaptive Elastic-Net method, and we compare it with LASSO and non-regularized (quasi-likelihood estimator). Adaptive estimators are based on the quasi-likelihood initial estimator, as in Remark \ref{rem:ini-ql}.

This next study shows the model selection capability of Elastic-Net in the
presence of correlation between the coordinates of a multivariate diffusion process.
Consider a stochastic regression model, i.e. a $(d+1)$-dimensional process $(Y, X_1, X_2, \ldots , X_d)$
that satisfies a SDE \eqref{eq:sde} with linear drift
\begin{equation}
b(y, x_1, , \ldots, x_d) = \left(
\begin{gathered}\label{eq:reg_model}
    \sum_{j=1}^d \alpha_j x_j - \alpha_{d+1} y  \\
    \alpha_{01} - \alpha_{11} x_1 \\
    \vdots  \\
    \alpha_{0d} - \alpha_{1d} x_d
\end{gathered}
\right)
\end{equation}
and diffusion matrix
\begin{equation}\label{eq:diff-corr}
A = 
\begin{bmatrix}
    \beta_0 & 0 & \cdots&0 \\
    0 & \beta_{11} & \cdots &  \beta_{1d} \\
    0 & \vdots &   &   \vdots \\
    0 & \beta_{d1} &\cdots  &  \beta_{dd}
\end{bmatrix}.
\end{equation}
Matrix $A$ controls the correlation between the regressors. We discuss separately the case $d=2$ from the setting $d>2.$

\textbf{Case} $\bm{d=2.}$
The parameters used in the simulation are as follows:
$(\alpha_1, \alpha_2, \alpha_3) = (1, 1, 1),$ $(\alpha_{01}, \alpha_{02}, \alpha_{11}, \alpha_{22}) = (0, 0, 1, 1)$, $\beta_0 = 1$ and
\[
\begin{aligned}
    &\begin{bmatrix}
        \beta_{11} & \beta_{12} \\
        \beta_{21} & \beta_{22}
    \end{bmatrix} = 
    \begin{bmatrix}
        0.8487 & 0.5316 \\
        0.5316 & 0.8487
    \end{bmatrix}
    = 
    \begin{bmatrix}
        1 & 0.9 \\
        0.9 & 1
    \end{bmatrix}^\frac 12 .
\end{aligned}
\]
The $\beta_{ij}$ values are chosen so that the noise correlation between $X_1$ and $X_2$ is $\rho = 0.9$.

\begin{figure}[ht!]
    \centering
    \includegraphics[width=0.75\linewidth]{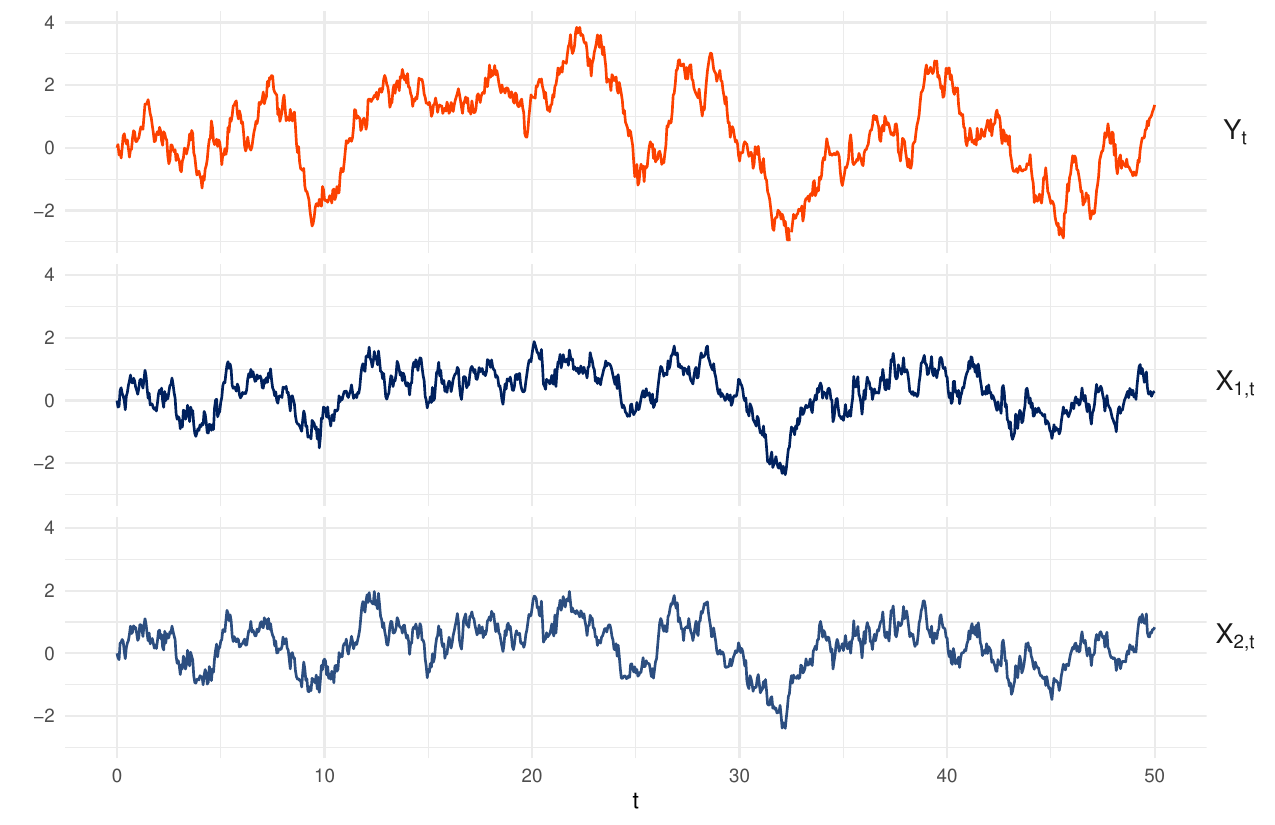}
    \caption{A sample path of the process $(Y, X_1, X_2)$ that satisfies a SDE with linear drift \eqref{eq:reg_model} and diffusion matrix \eqref{eq:diff-corr}.}
    \label{fig:enter-label}
\end{figure}

We consider the metrics \emph{accuracy} and \emph{selection}. Accuracy measures proportion of parameters correctly included/excluded, selection measures the probability of including all of the true predictors in the regression equation. In the $d=2$ case both the true coefficients $\alpha_1$ and $\alpha_2$ of $X_1$ and $X_2$ respectively are non-zero.

The simulation setting is as follows. We repeated $N=10^3$ simulations from the process $(Y, X_1, X_2)$ with linear drift \eqref{eq:reg_model} and diffusion matrix \eqref{eq:diff-corr}, with increasing sample sizes equal to $n=250$, $n=1000$ and $n=10\,000$, according to a high-frequency sampling scheme in order to approach the asymptotic regime. The mixing parameter $\gamma$ is set equal to 0.5. The optimal lambda value, $\lambda_{opt}$, is chosen based on an AIC criterion. Results are rounded up to the second decimal point. In this study the diffusion parameters are fixed, we focus on the drift estimation in order to isolate the impact of correlation on variable selection, in analogy with linear regression.  
 
Results are shown in \autoref{tab:stoch-reg}. We see that Elastic-Net regularization is capable of correctly dealing with the coefficients of the correlated variables, by including both of them, while simultaneously selecting the correct model. The LASSO instead tends to select only one of the two variables, with a lower percentage of model identification. This effect is most apparent for a moderate sample size. 

\begin{table}[ht!]
\centering
\caption{Stochastic Regression Simulation Results}
\label{tab:stoch-reg}
\small
\begin{adjustbox}{max width=1\textwidth,center}
\begin{tabular}{cccccccccc}
\hline
\textbf{$n$} & \textbf{$\Delta_n$} &  \textbf{$T$} & & \textbf{E-Net Accuracy} & \textbf{E-Net Selection} & & \textbf{LASSO Accuracy} & \textbf{LASSO Selection} \\
\cline{5-6}\cline{8-9}
250   & 0.1  & 25 & & 0.90 & 0.73 & & 0.88 & 0.27 \\
1000  & 0.05 & 50 & & 0.93 & 0.89 & & 0.92 & 0.6 \\
10\,000 & 0.01 & 100 & & 0.97 & 0.98 && 0.97 & 0.87 \\
\hline
\end{tabular}
\end{adjustbox}
\end{table}

Let $\hat{\alpha}_{n}^{(k)}$ denote the estimate obtained at simulation $k$ for the drift parameter whose true values are denoted by $\alpha_0$. We evaluate the performance of the Elastic-Net estimator LASSO estimator by computing the empirical mean square errors: \begin{equation}
\widehat{\mathsf{MSE}}_j = \frac{1}{N} \sum_{k=1}^N (\hat{\alpha}_{n,j}-\alpha_{0,j})^2 \qquad \text{for } j=1,\ldots,\p   
\end{equation}

The results obtained with $N=10^3$ simulations, as presented in \autoref{tab:results-distr}, demonstrate that the estimates generated by Elastic-Net method exhibit superior performance in terms of empirical mean square error compared to those produced by the LASSO across all three scenarios. It is noteworthy that the estimates derived using the Elastic-Net exhibit a lower mean square error for the parameters $\alpha_1$ and $\alpha_2$, which, as shown in \eqref{eq:reg_model}, are associated with the correlated variables, particularly in the case of small sample sizes. 
%As the sample size increases, the QMLE estimator also achieves a low empirical MSE, however, it does not perform parameter selection, as it does not shrink null parameters to zero.

To provide a clear view of the differences in the estimates obtained with Elastic-Net and LASSO, \autoref{fig:distrPar} shows the empirical distributions of the parameters $\alpha_1$ and $\alpha_2$. 

In adherence to the theoretical framework, which asserts that LASSO tends to select only one of the correlated variables, unlike Elastic-Net, which does not shrink the parameters associated to correlated variables to zero, and in alignment with the findings depicted in \autoref{tab:stoch-reg} and \autoref{tab:results-distr}, the plots show that Elastic-Net exhibits lower percentages of parameter shrinkage across all three scenarios. This phenomenon is especially discernible in scenarios involving smaller sample sizes.

\autoref{tab:results-distr} provides a summary of the results obtained from the simulation study.

\begin{figure}[ht]
	\centering
	\begin{subfigure}[b]{0.5\textwidth}
		\includegraphics[width=\textwidth]{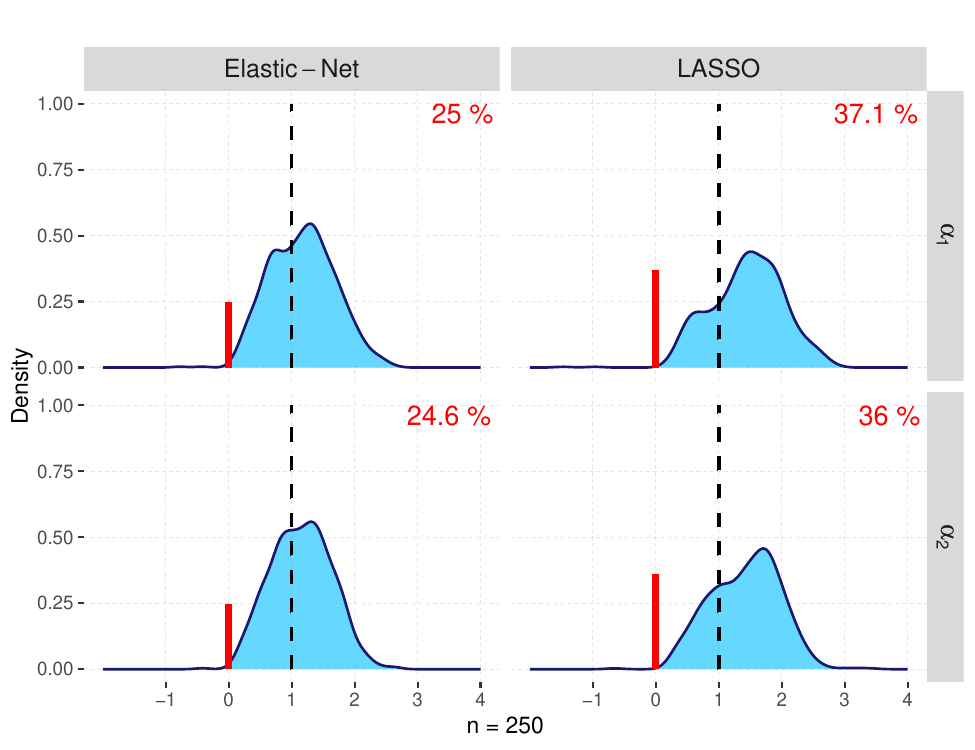}
	\end{subfigure}
	\begin{subfigure}[b]{0.5\textwidth}
		\includegraphics[width=\textwidth]{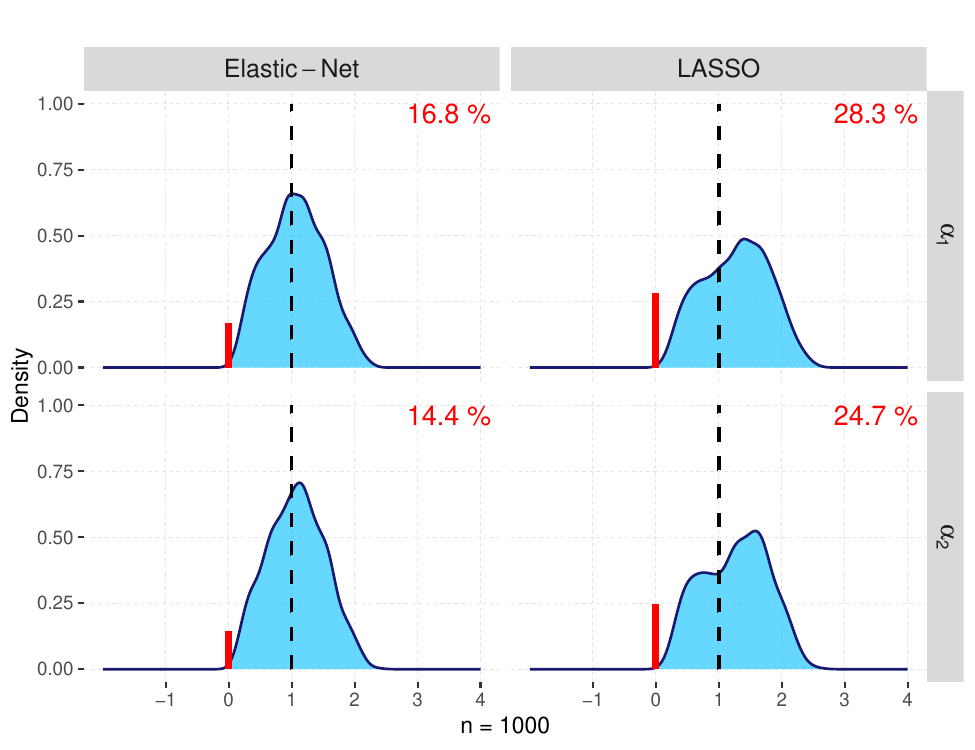}
	\end{subfigure}
	\begin{subfigure}[b]{0.5\textwidth}
		\includegraphics[width=\textwidth]{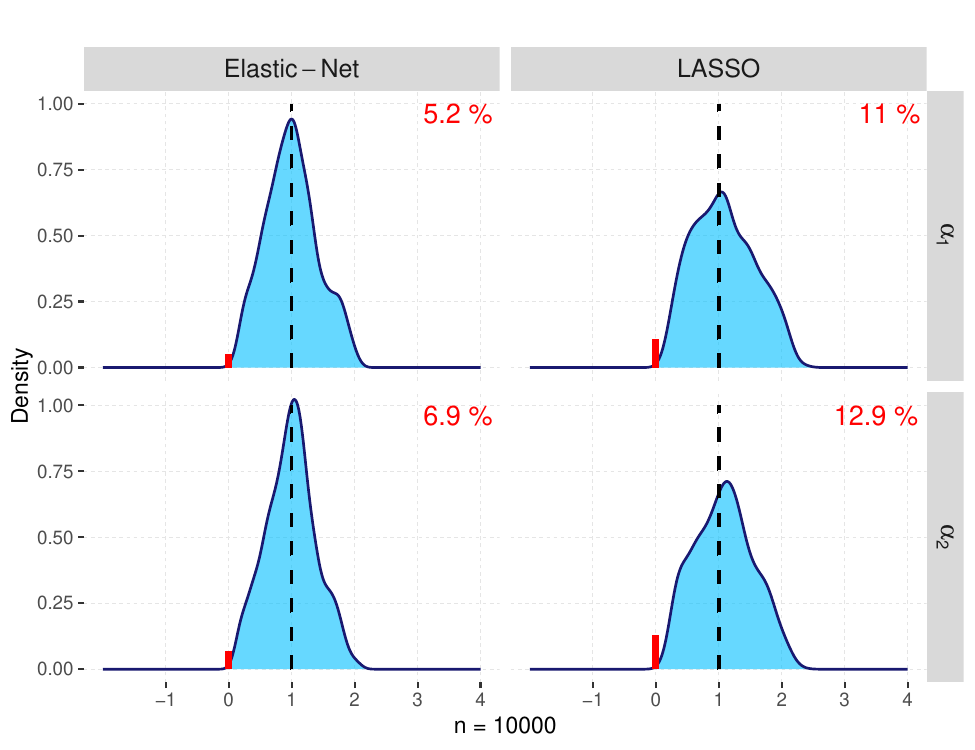}
	\end{subfigure}
	\caption{Distribution of parameters $\alpha_1$ and $\alpha_2$, with increasing sample size. For each parameter-method combination, the plot shows in red the proportion of parameters estimated as zero by the method under consideration, while in blue the density of parameters estimated as non-zero. The true parameter values are shown with black dashed lines. The integration of the continuous and discrete components in each plot is normalized to 1.}
	\label{fig:distrPar}
\end{figure}

\begin{table}[ht!]
    \centering
    \caption{Summary of estimated parameters across different sample sizes $n$.}
    \subcaptionbox{$n=250$}{
    \begin{adjustbox}{max width=1\textwidth,center}
    \begin{tabular}{lcccccccccccccc}
    \hline
       Par.  & \multicolumn{1}{l}{True} & \multicolumn{5}{l}{Elastic-Net} & & \multicolumn{5}{l}{LASSO} \\
        
    \cline{3-7}\cline{9-13}
             &     \multicolumn{1}{l}{Value}   &  $\widehat{MSE}$    & \raisebox{3pt}{$\underset{(Std. Err.)}{Avg.}$} & $q_{0.25}$ & $Me$ & $q_{0.75}$ && $\widehat{MSE}^{\phantom{\hat{a}}}$  & \raisebox{3pt}{$\underset{(St. Err.)}{Avg.}$} & $q_{0.25}$ & $Me$ & $q_{0.75}$ \\
             \hline
 $\alpha_1$ & 1 & 0.469 & $\underset{(0.674)}{0.876}$  & 0 & 0.911 & 1.392 && 0.705 & $\underset{(0.835)}{0.909}$  & 0 & 0.943 & 1.640 \\
 $\alpha_2$  & 1 & 0.436 & $\underset{(0.647)}{0.867}$ & 0.167 & 0.922 & 1.369 &&  0.664  & $\underset{(0.810)}{0.909}$ & 0 & 0.945 & 1.639 \\
 $\alpha_3$  & 1 & 0.052 & $\underset{(0.210)}{0.910}$  & 0.765 & 0.893 & 1.039 &&   0.060  &$\underset{(0.234)}{0.926}$ & 0.759 & 0.907 & 1.069 \\
 $\alpha_{01}$ & 0 & 0.011 & $\underset{(0.103)}{0.004}$  & 0 & 0 & 0 && 0.007  & $\underset{(0.086)}{0.001}$  & 0 & 0 & 0 \\
 $\alpha_{02}$ & 0 & 0.011 & $\underset{(0.104)}{0.002}$  & 0 & 0 & 0 && 0.007 & $\underset{(0.085)}{0}$ & 0 & 0 & 0 \\
 $\alpha_{11}$ & 1 & 0.060 & $\underset{(0.235)}{0.929}$ & 0.768 & 0.906 & 1.070 && 0.073 & $\underset{(0.256)}{0.912}$  & 0.731 & 0.889 & 1.070 \\
 $\alpha_{12}$ & 1 & 0.058 & $\underset{(0.230)}{0.928}$ & 0.768 & 0.913 & 1.071 &&  0.070  & $\underset{(0.249)}{0.911}$ & 0.741 & 0.884 & 1.068 \\
 \hline 
 \multicolumn{13}{c}{} 
    \end{tabular}\end{adjustbox}}
\subcaptionbox{$n=1000$}{
    \begin{adjustbox}{max width=1\textwidth,center}
    \begin{tabular}{lcccccccccccc}
    \hline
       Par.  & \multicolumn{1}{l}{True} & \multicolumn{5}{l}{Elastic-Net} & & \multicolumn{5}{l}{LASSO}  \\
    \cline{3-7}\cline{9-13} 
             &     \multicolumn{1}{l}{Value}   &  $\widehat{MSE}$     & \raisebox{3pt}{$\underset{(Std. Err.)}{Avg.}$} & $q_{0.25}$  & $Me$ & $q_{0.75}$ && $\widehat{MSE}^{\phantom{\hat{a}}}$  & \raisebox{3pt}{$\underset{(St. Err.)}{Avg.}$} &  $q_{0.25}$ & $Me$ & $q_{0.75}$ \\
             \hline
 $\alpha_1$ & 1 &  0.346 & $\underset{(0.577)}{0.885}$ & 0.432 & 0.939 & 1.314 &&  0.531  & $\underset{(0.724)}{0.911}$  &  0 & 0.963 & 1.528 \\
 $\alpha_2$ & 1 & 0.323 & $\underset{(0.563)}{0.919}$  & 0.525 & 0.974 & 1.334 &&   0.502  & $\underset{(0.707)}{0.947}$  & 0.251 & 1.003 & 1.550 \\
 $\alpha_3$ & 1 & 0.028 & $\underset{(0.153)}{0.929}$ & 0.820 & 0.916 & 1.022 && 0.031  & $\underset{(0.167)}{0.947}$ & 0.831 & 0.935 & 1.048 \\
 $\alpha_{01}$ & 0 &  0.001 & $\underset{(0.038)}{0}$  & 0 & 0 & 0 && 0.001  & $\underset{(0.037)}{0.001}$  & 0 & 0 & 0 \\
$\alpha_{02}$ & 0 &  0.001 & $\underset{(0.036)}{0}$  & 0 & 0 & 0 && 0.001  & $\underset{(0.035)}{0}$  & 0 & 0 & 0 \\
 $\alpha_{11}$ & 1 & 0.027 & $\underset{(0.153)}{0.938}$  & 0.834 & 0.926 & 1.038 &&  0.031  & $\underset{(0.166)}{0.942}$ & 0.825 & 0.928 & 1.050 \\
 $\alpha_{12}$  & 1 & 0.028 & $\underset{(0.155)}{0.938}$  & 0.832 & 0.925 & 1.029  && 0.031  & $\underset{(0.168)}{0.942}$  & 0.825 & 0.927 & 1.044 \\
 \hline
 \multicolumn{13}{c}{} 
    \end{tabular}\end{adjustbox}}
    \subcaptionbox{$n=10\,000$}{
    \begin{adjustbox}{max width=1\textwidth,center}
    \begin{tabular}{lcccccccccccc}
    \hline
       Par.  & \multicolumn{1}{l}{True} & \multicolumn{5}{l}{Elastic-Net} & & \multicolumn{5}{l}{LASSO}  \\
    \cline{3-7}\cline{9-13} 
             &     \multicolumn{1}{l}{Value}   &  $\widehat{MSE}$     & \raisebox{3pt}{$\underset{(Std. Err.)}{Avg.}$} & $q_{0.25}$ & $Me$ & $q_{0.75}$ && $\widehat{MSE}^{\phantom{\hat{a}}}$  & \raisebox{3pt}{$\underset{(St. Err.)}{Avg.}$} & $q_{0.25}$ & $Me$ & $q_{0.75}$ \\
             \hline
 $\alpha_1$ & 1 & 0.214 & $\underset{(0.459)}{0.947}$  & 0.643 & 0.958 & 1.244 & & 0.331 & $\underset{(0.575)}{0.967}$ & 0.539 & 0.980 & 1.400 \\
 $\alpha_2$ & 1 & 0.212 & $\underset{(0.454)}{0.922}$ & 0.634 & 0.964 & 1.207 & & 0.324 &   $\underset{(0.567)}{0.945}$  & 0.524 & 0.994 & 1.341  \\
$\alpha_3$ & 1 & 0.012 & $\underset{(0.098)}{0.950}$  & 0.884 & 0.940 & 1.013 & & 0.012 &  $\underset{(0.104)}{0.970}$  & 0.898 & 0.961 & 1.036 \\
 $\alpha_{01}$ & 0 & $<10^{-3}$ & $\underset{(0.020)}{0}$  & 0 & 0 & 0 &  & $<10^{-3}$ & $\underset{(0.021)}{0}$ & 0 & 0 & 0 \\
 $\alpha_{02}$ & 0 & $<10^{-3}$ & $\underset{(0.022)}{0.001}$  & 0 & 0 & 0 & & 0.001 & $\underset{(0.023)}{0.001}$  & 0 & 0 & 0  \\
$\alpha_{11}$ & 1 & 0.013 &  $\underset{(0.107)}{0.959}$  & 0.885 & 0.956 & 1.027 & & 0.013  &  $\underset{(0.113)}{0.974}$  & 0.895 & 0.970 & 1.045  \\
 $\alpha_{12}$ & 1 & 0.012 & $\underset{(0.104)}{0.962}$  & 0.889 & 0.963 & 1.030 & & 0.012 &  $\underset{(0.109)}{0.977}$ & 0.901 & 0.973 & 1.048 \\
 \hline
    \end{tabular}\end{adjustbox}}
    \label{tab:results-distr}
\end{table}

\textbf{Case} $\bm{d >2}$. In this scenario, we consider the case where some of the coefficients of processes $X_1, \ldots, X_d$ are zero. 
We assume the following correlation structure:
$
\Sigma=(\Sigma_{ij})_{i,j} = (\rho^{|i - j|})_{i,j}
$
for $\rho = 0.5, 0.8, 0.9$. 
The drift parameters used in the simulation are
$$(\alpha_1, \ldots, \alpha_{\lfloor d/2\rfloor}, \alpha_{\lfloor d/2\rfloor+1}, \ldots, \alpha_d, \alpha_{d+1}) = (1/d, \ldots, 1/d, 0, \ldots, 0, 2),$$ with $\alpha_{0,j}=1, \alpha_{1,j} = 1, j \in 1, \ldots, d$.
Results are shown in figure \autoref{fig:rho-mixing}. 

\textit{Effect of the Elastic-Net parameter}. From \autoref{fig:rho-mixing} we see how the Elastic-Net parameter $\gamma$ controls the trade-off in variable selection. For values of $\gamma$ closer to 1, i.e. the LASSO case,  the probability of correctly selecting the model decreases drastically due to the presence of correlated processes that the model cannot properly distinguish, especially for larger values of $\rho$. 
Conversely for smaller values of the Elastic-Net parameter $\gamma$ this effect can be controlled. The cost is a smaller ``accuracy", due to a higher number of variables falsely included in the model. 

\begin{figure}[ht!]
    \centering
    \includegraphics[width=0.6\linewidth]{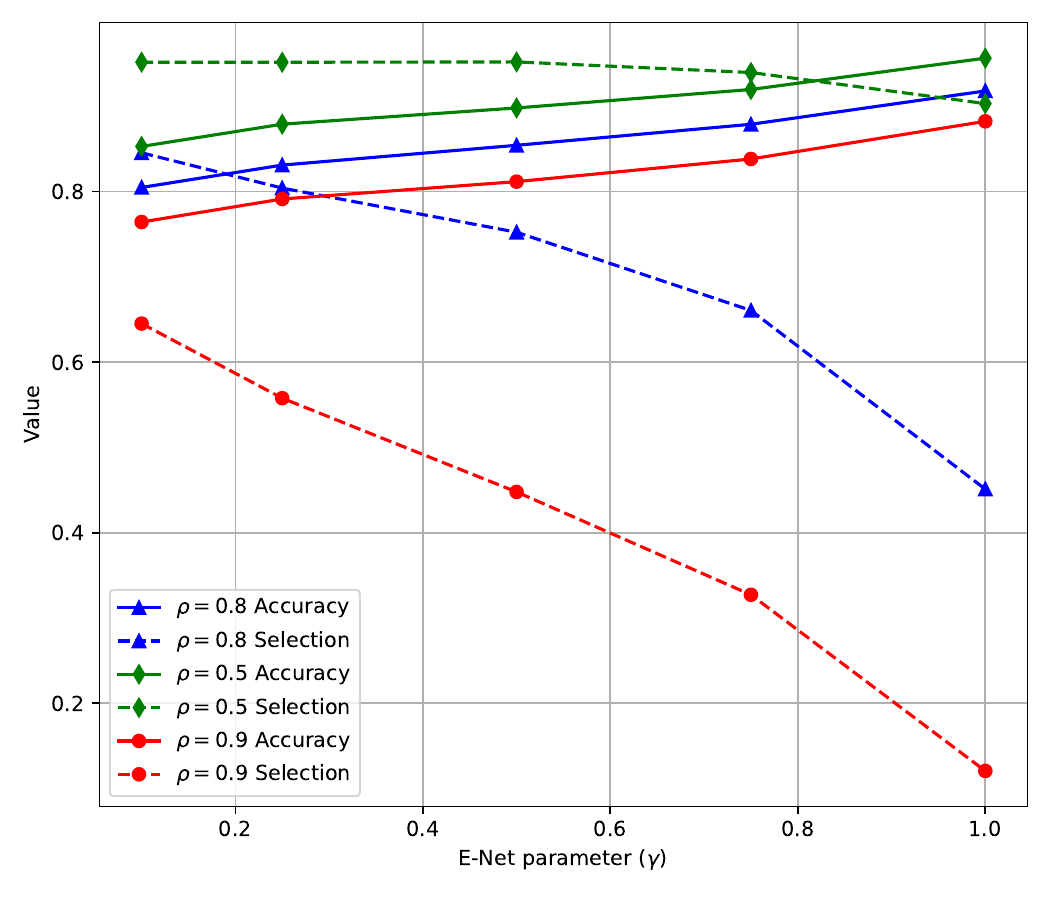}
    \caption{Accuracy and selection for different values of the mixing parameter and correlation parameter.}
    \label{fig:rho-mixing}
\end{figure}

\textit{Forecasting performance.}
We now turn our attention to assessing the predictive efficacy of the Elastic-Net estimator. To this end, we conduct $N=10^4$ simulations from the model specified in \eqref{eq:reg_model} and \eqref{eq:diff-corr}, considering the scenarios outlined in \autoref{tab:stoch-reg}. We observe the data up to time $T_n = T$, while our goal is to predict future values up to time $T+1$, with a step size of $\Delta_n\in\{0.05, 0.01\}$. For each simulation $k$, we estimate $X_{T+h}^{(k)}$ at time $T+h$, denoted by $\hat{X}_{T+h}^{(k)}$, is performed using the one-step-ahead prediction as described in \eqref{eq:onestep-pred}, with $h$ varying from 0 to 1, i.e. $h\in\{0,\Delta_n, 2\Delta_n, \ldots, 1\}$. We can calculate the empirical mean absolute prediction error for the Elastic-Net and LASSO estimators as follows:
\begin{equation}\label{empir-mae}
    \widehat{\mathsf{MAE}}_k = \frac{1}{N} \sum_{k=1}^N |X^{(k)}_{T+h}-\hat{X}^{(k)}_{T+h}| \qquad \text{with } h\in\{0, \Delta_n, 2\Delta_n, \ldots, 1\}.
\end{equation}

Similarly, the non-asymptotic prediction error bound, at each time point $T+h$, can be computed in a straightforward manner up to constant factors, as described in \eqref{eq:pred-bound-intuition}. This allows for comparison with the empirical mean absolute prediction error calculated for the Elastic-Net, LASSO and QMLE estimators.  The results are shown in the graphs in \autoref{fig:maemore}. In the first two graphs, \autoref{fig:sub1-maemore2} and \autoref{fig:sub2-maemore2}, the number of parameters is the same, $\p=7$, while the sample size $n$ and the step size $\Delta_n$ vary between the two figures. 
The approximate prediction error bound (dashed line), calculated according to \eqref{eq:pred-bound-intuition}, thus up to constant factors, correctly predicts the shape of the empirical mean absolute prediction error \eqref{empir-mae} for the Elastic-Net Estimator. It can be noted that for very large sample sizes the error curves align, consistently with the expected asymptotic behaviour of each estimator. Figures \ref{fig:sub3-maemore2} and \ref{fig:sub4-maemore2}, show the result for the same sample size $n$, and increasing number of parameters $\p$. It can be observed that the gap in predictive performance widens, with the Elastic-Net estimator being more stable as $\p$ and the prediction horizon $h$ increase. Our results show that the introduction of an Elastic-Net tunable penalty can considerably improve our ability to forecast a correlated dynamical system, while the LASSO penalty proves to be too rigid.

\begin{figure}[ht!]
	\centering
    \begin{subfigure}[b]{0.49\textwidth}
		\includegraphics[width=\textwidth]{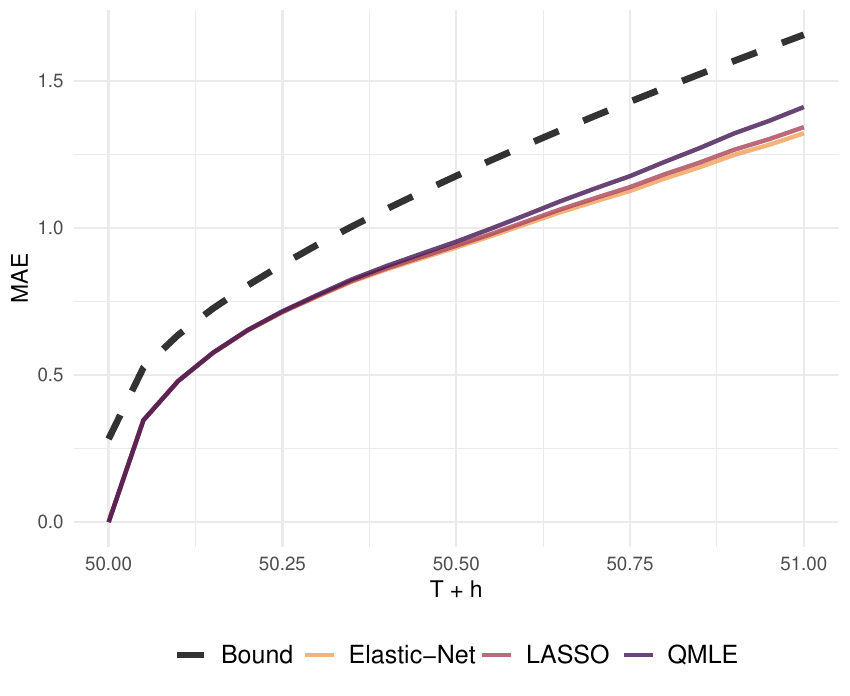}
		\caption{$n = 1000, \Delta_n=0.05 , d=2, \p = 7$}
		\label{fig:sub1-maemore2}
	\end{subfigure}
	\hfill 
	\begin{subfigure}[b]{0.49\textwidth}
		\includegraphics[width=\textwidth]{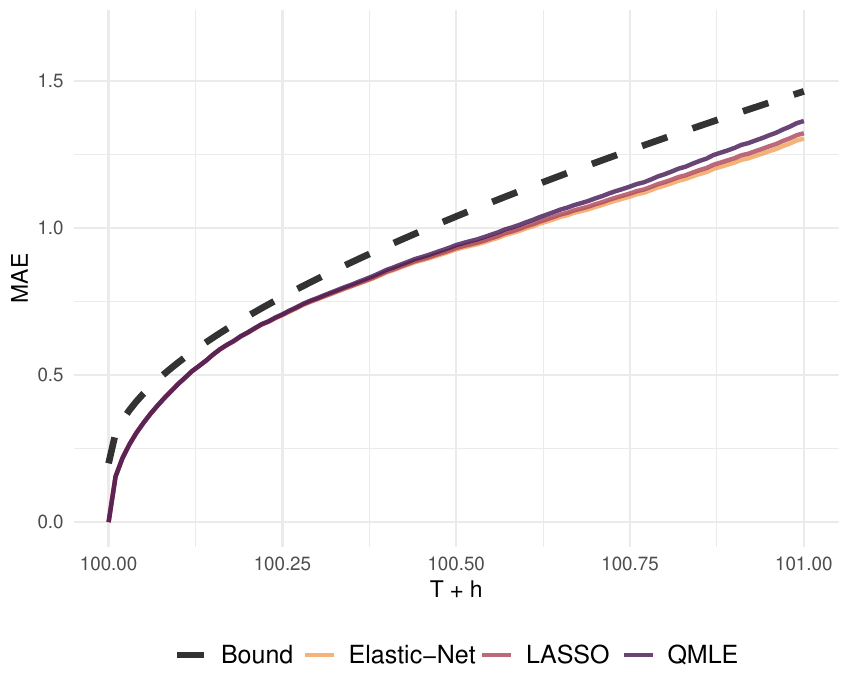}
		\caption{$n = 10\,000, \Delta_n=0.01 , d=2, \p = 7$}
		\label{fig:sub2-maemore2}
	\end{subfigure}
	\hfill 
	\begin{subfigure}[b]{0.49\textwidth}
		\includegraphics[width=\textwidth]{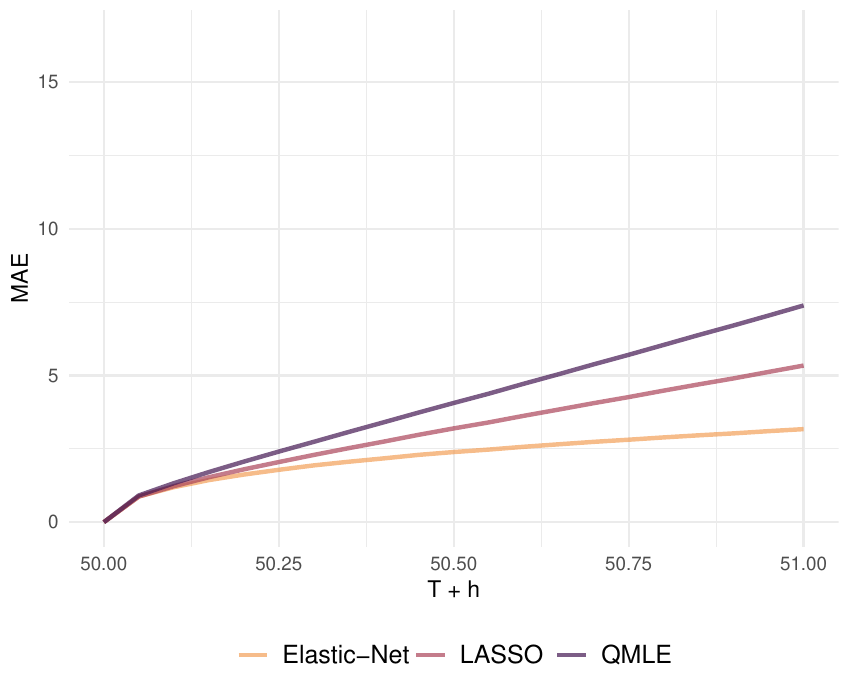}
		\caption{$n = 1000, \Delta_n=0.05 , d=16, \p = 49$}
		\label{fig:sub3-maemore2}
	\end{subfigure}
	\hfill 
	\begin{subfigure}[b]{0.49\textwidth}
		\includegraphics[width=\textwidth]{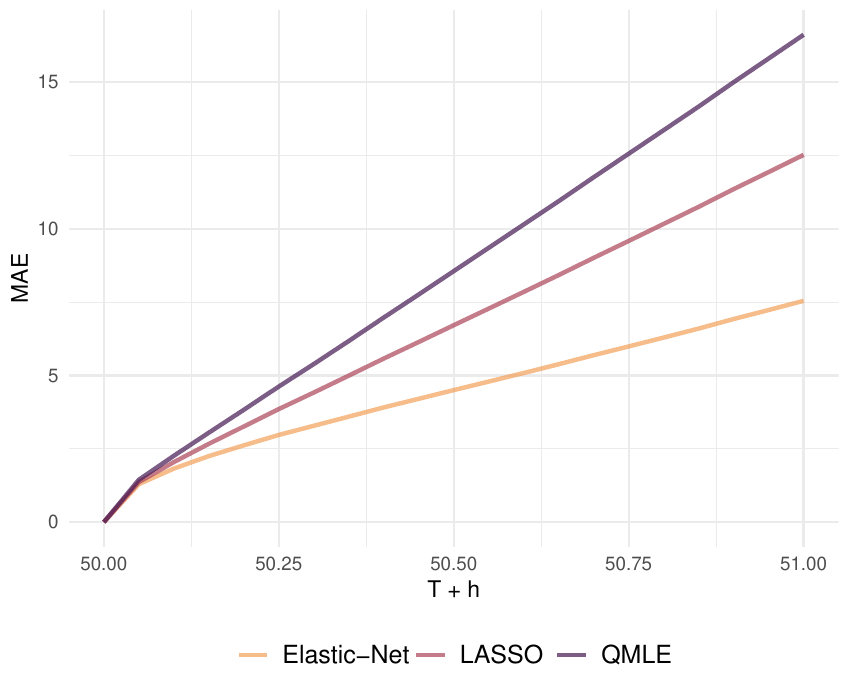}
		\caption{$n = 1000, \Delta_n=0.05 , d=33, \p = 100$}
		\label{fig:sub4-maemore2}
	\end{subfigure}
	\caption{Comparison of empirical mean absolute prediction errors for Elastic-Net, LASSO and QMLE estimators with respect of variation in the prediction horizon $h$, sample size $n$ and the number of parameters. The graphs in the same row share the same scale on the $y$-axis. In \autoref{fig:sub1-maemore2} and \autoref{fig:sub2-maemore2}, the non-asymptotic prediction error bound, calculated according to \eqref{eq:pred-bound-intuition}, up to constant factors, is also represented by dashed line.}
	\label{fig:maemore}
\end{figure}

\section{Real data application: well-being data analysis}\label{sec:real}
In this section, we use our technique to discover association dependency patterns in real-world scenarios.  
Motivated by \cite{carpi2022}, we apply our methodology to the analysis of subjective well-being. In that study, the authors examine the effect of the COVID-19 pandemic on subjective well-being (SWB), assessed through Twitter data from Japan and Italy. The study analyzes various data sources, including climate and air quality, COVID-19 cases and deaths, survey data on symptoms, Google search trends, policy measures, mobility patterns, economic indicators, and proxies for health and stress. We refer to the original study for a complete description of the dataset.
In \autoref{fig:wb-data} we show the response variable and one of the predictors included in the model. 

\begin{figure}[ht!]
	\centering
	\begin{subfigure}[b]{0.5\textwidth}
		\includegraphics[width=\textwidth]{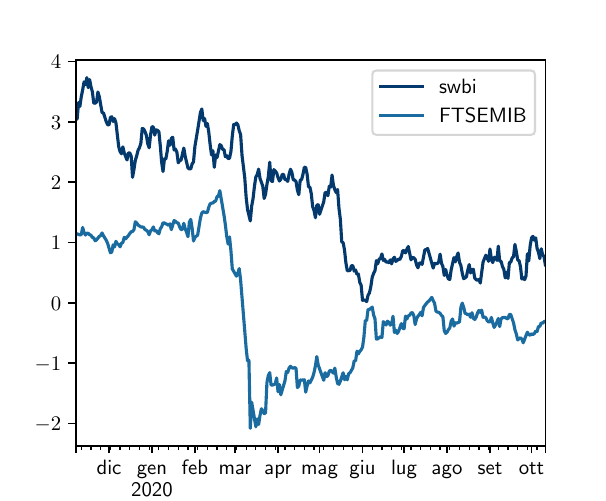}
		\caption{Italy}
	\end{subfigure}
	\hfill % This adds a horizontal space between the images
	\begin{subfigure}[b]{0.45\textwidth}
		\includegraphics[width=\textwidth]{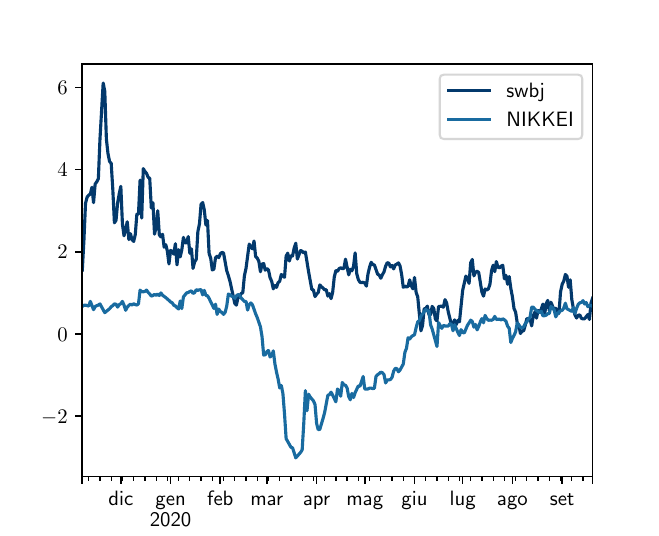}
		\caption{Japan}
	\end{subfigure}
	\caption{Response variable (subjective well being), and one predictor (country's economic index).}
	\label{fig:wb-data}
\end{figure}

We split the data in shorter 1-month or 3-months time windows, and we carry out an Elastic net analysis on each time frame. Such a framework, called Dynamic Elastic Net, allows to mitigate the effect of potential non-stationarity in the data (see \cite{carpi2022}). 
The goal is to identify which variables contribute to the SWB and to analyze how the association patterns change over time.  
We finally compare the variables in the Importance-Frequency space in order to assess their overall relative importance. 

% The workflow of the analysis is shown in \autoref{fig:workflow}.

% \begin{figure}[ht!]
%     \centering
%     \includegraphics[width=0.8\textwidth]{enet-analysis.png}
%     \caption{Analysis workflow}
%     \label{fig:workflow}
% \end{figure}

%\textbf{Relative coefficients.}
We treat the SWB variable as a response $Y$ in a stochastic regression model
and the other variables as predictors $X_1, \ldots, X_p$. 
We then consider a regression-type SDE of the form
\begin{gather}
    \de Y_t = \left( \mu_Y - \alpha_{Y,0} Y_t - \sum_{j=1}^p \alpha_{Y,j} X_{j,t} \right) \de t + \sigma_Y \sqrt{Y_t} \de W_{1,t}
    \\
    \de X_{j,t} = \left( \mu_j - \alpha_{X,j} X_{j,t} \right) \de t + \sigma_j \de W_{j+1,t}
\end{gather}
where $W_1, \ldots, W_{j+1}$ are independent Brownian motions and $X_j$ are Vasicek SDE. 
We use Elastic-Net estimation in order to select the true non-zero coefficients in the model. More in detail, we compute an initial quasi-likelihood estimate for all the parameters in the model, and then we apply our adaptive estimator by restricting our attention to the parameters $\alpha_{Y,j}, j=1, \ldots p$. We repeat this process for subsequent 30 and 90 days windows, and we normalize the absolute value of the coefficients in order to have comparable results over different sets of predictors. We run our model on both Italy and Japan data. 
The results for a 90 days time frame are shown in \autoref{fig:swb-coef}.

\begin{figure}[ht!]
	\centering
	\begin{subfigure}[b]{0.9\textwidth}
		\includegraphics[width=\textwidth]{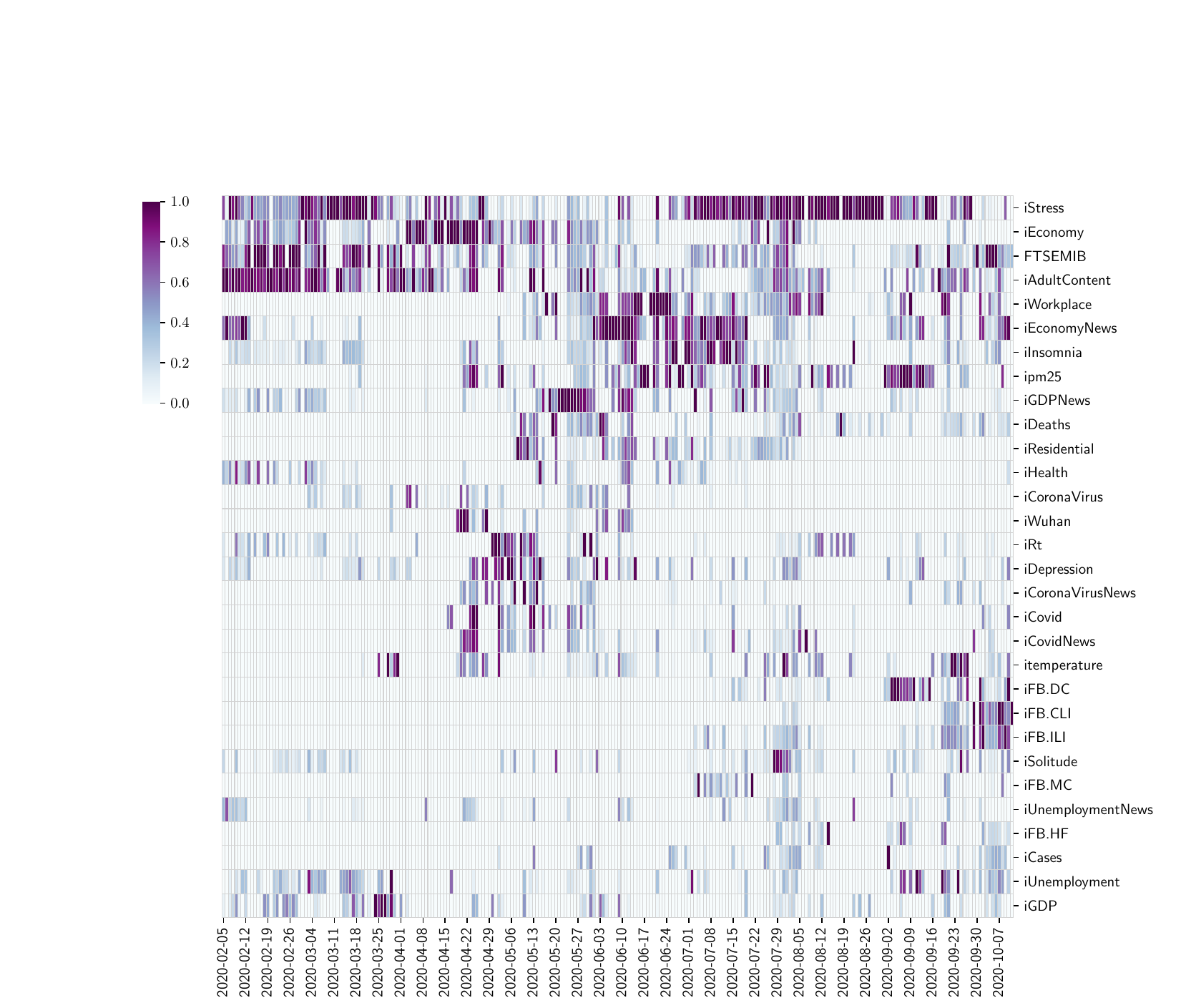}
		\caption{Italy}
	\end{subfigure}
	\begin{subfigure}[b]{0.9\textwidth}
		\includegraphics[width=\textwidth]{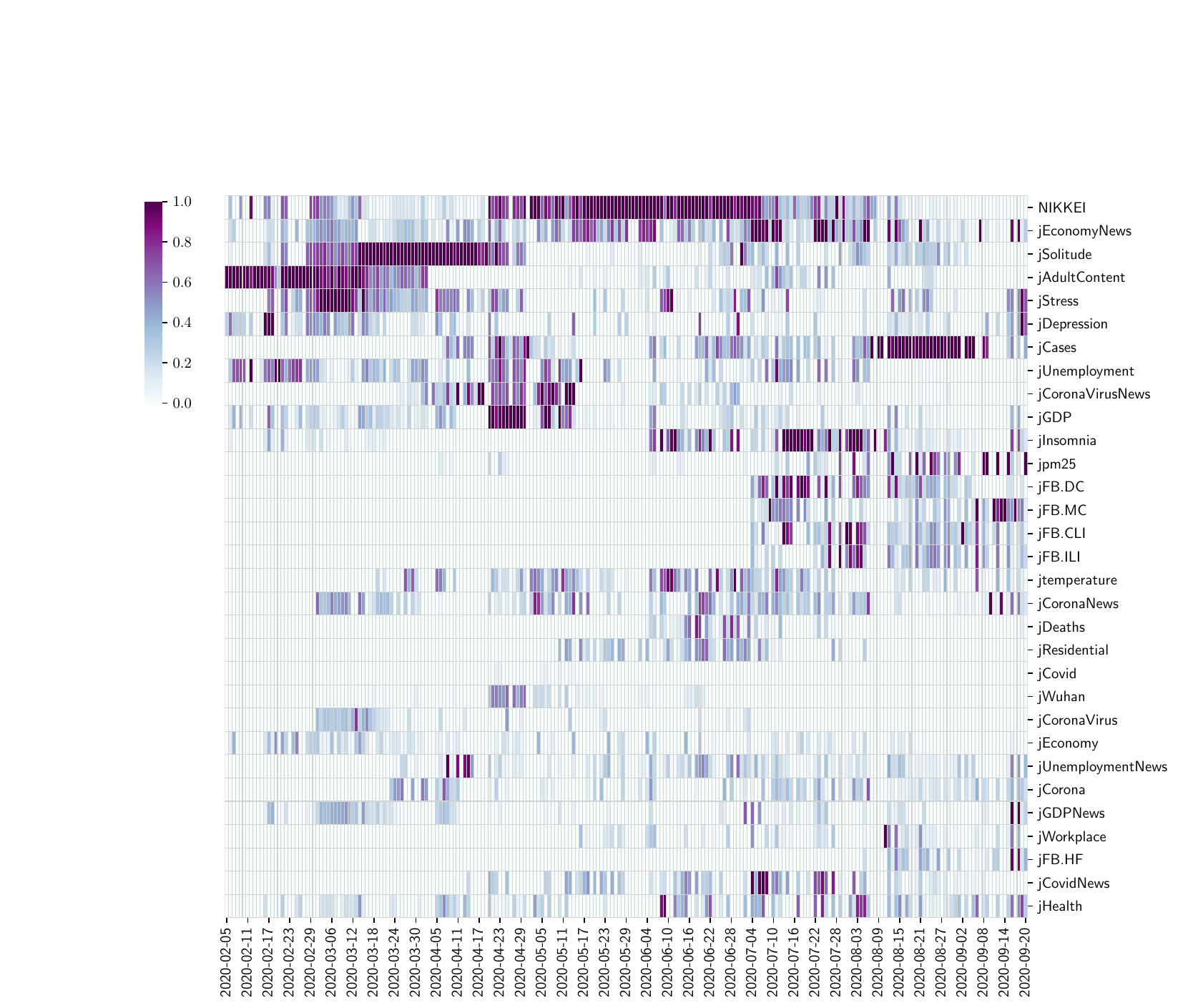}
		\caption{Japan}
	\end{subfigure}
	\hfill % This adds a horizontal space between the images
	\caption{Normalized coefficients $\alpha_{Y,j}$, 90 days time window.}
	\label{fig:swb-coef}
\end{figure}

Finally, we compare the results in the Importance-Frequency space. The frequency is given by the number of times that a variable is selected over the different time windows. The importance component is computed by first ranking the selected variables in each time window, and then averaging the results (normalized in 0-100).  
The results are shown in \autoref{fig:swb-imp}. 
Variables in the upper right corner have a stronger effect on the response, in the sense that they are selected a larger number of times and, whenever they are selected, they tend to have the largest effect. Vice versa variables placed in the bottom left corner are comparatively less important. 

We note how in Italy, \autoref{fig:swb-imp} (A), in the shorter time frame (31 days), variables related to mobility restrictions (residential) or to the pandemic spread (deaths, Rt) tend to be more important than in the long time. In the long run (90 days) psychological stress and economy have a larger impact. 
When comparing Italy and Japan, \autoref{fig:swb-imp} (B), we note that in Japan the number of Covid cases (jCases) have a larger effect, while in Italy deaths (iDeaths) are more impactful. Economy variables are generally more relevant in Japan.

\textbf{Code availability.} The code for the numerical analysis presented in this paper is publicly available at online repository \cite{enet-sde}.

\begin{figure}[ht!]
	\centering
	\begin{subfigure}[b]{0.8\textwidth}
		\includegraphics[width=\textwidth]{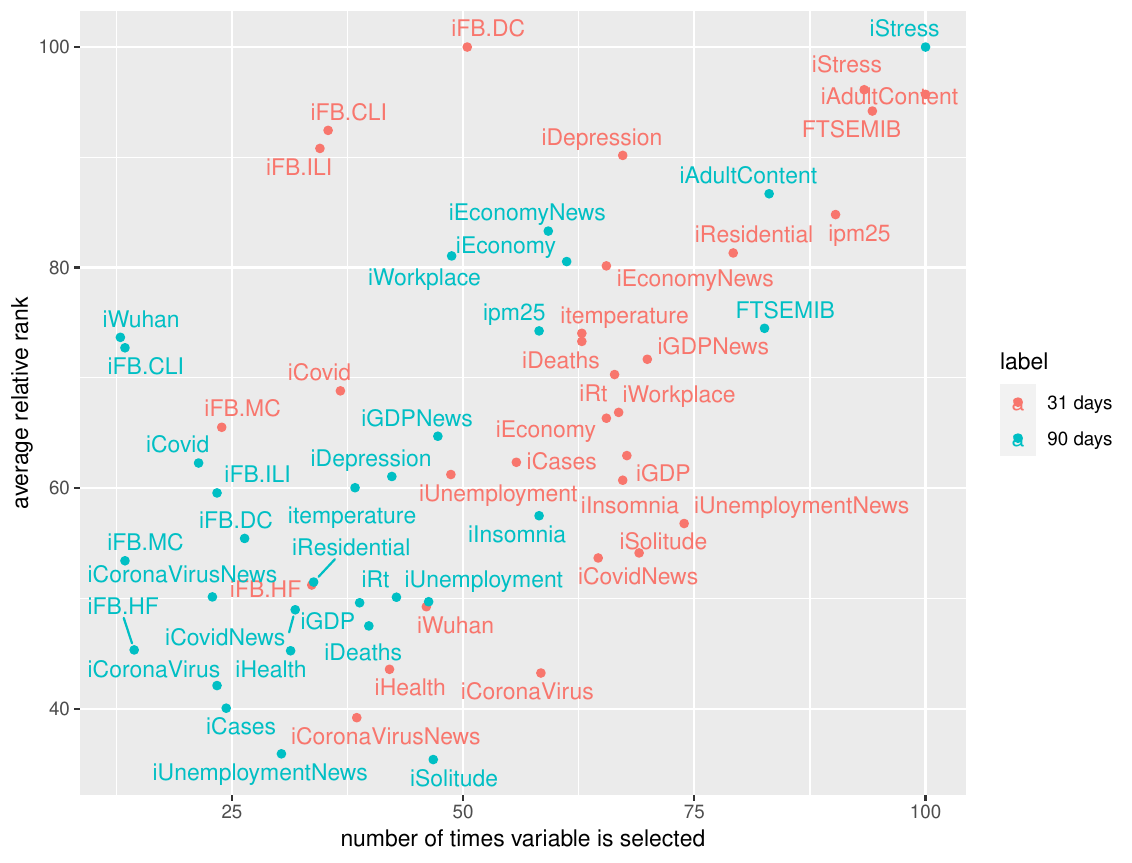}
		\caption{Italy data, comparison between 31 and 90 days period}
	\end{subfigure}
	\begin{subfigure}[b]{0.8\textwidth}
		\includegraphics[width=\textwidth]{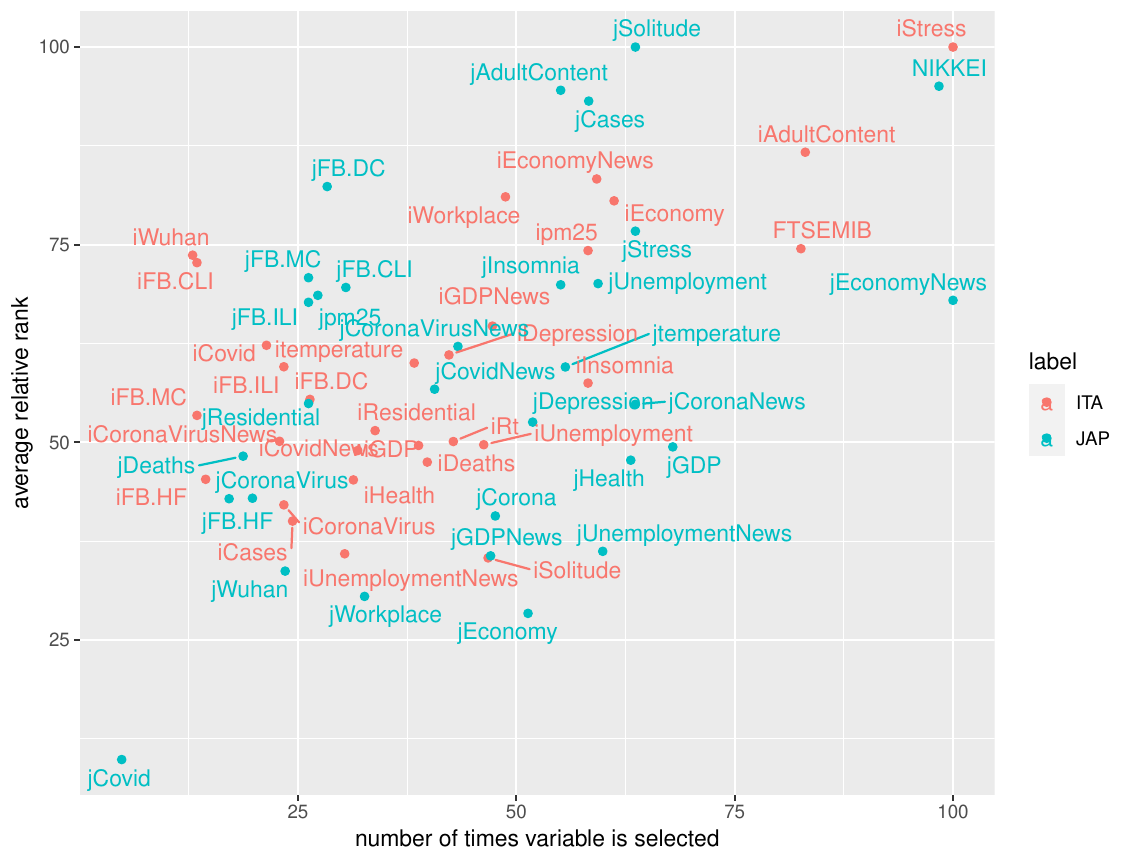}
		\caption{90 days time frame, comparison between Italy and Japan results}
	\end{subfigure}
	\hfill % This adds a horizontal space between the images
	\caption{Variables representation in the frequency-importance space.}
	\label{fig:swb-imp}
\end{figure}

\newpage

\section{Proofs}
\begin{proof}[Proof of Theorem \ref{thm:asoracle}] In the proof we use the same approach adopted in \cite{de2021regularized}. In the present setting it is necessary to handle the Ridge penalty.

    For the sake of simplicity, some of the dependency on $n$, indicated as a subscript, will not be taken into account in the following steps (i.e. $\hat{\theta}=\hat{\theta}_n$, $\tilde{\theta}=\tilde{\theta}_n$, $A=A_n$, $\hat{G}=\hat{G}_n$, $\hat{D}=\hat{D}_n$, $\kappa_j=\kappa_{n,j}$,  $\pi_j=\pi_{n,j}$).
    
    (i) Let us start with the consistency. We can write down
    \begin{align*}
        0 \geq & \mathcal{F}_n(\hat{\theta},\tilde\theta)-\mathcal{F}_n(\theta_0, \tilde\theta)\\
       %= & \ (\hat{\theta}-\tilde{\theta})^\tr \hat{G}(\hat{\theta}-\tilde{\theta})+\sum_{j=1}^{\p}\kappa_{j}|\hat{\alpha}_j|+\sum_{h=1}^{\q} \pi_{h} |\hat{\beta}_h|+\lambda_{n}\sum_{j=1}^{\p}|\hat{\alpha}_j|^2+\gamma_{n}\sum_{h=1}^{\q}|\hat{\beta}_h|^2\\
       %& - \left((\theta_0-\tilde{\theta})^\tr \hat{G}(\theta_0-\tilde{\theta})+\sum_{j=1}^{\p}\kappa_{j}|\alpha_{0,j}|+\sum_{h=1}^{\q} \pi_{h} |\beta_{0,h}|+\lambda_{n}\sum_{j=1}^{\p}|\alpha_{0,j}|^2+\gamma_{n}\sum_{h=1}^{\q}|\beta_{0,h}|^2\right)\\
       = & \langle \hat G, (\hat\theta-\theta_0)^{\otimes 2}\rangle+2(\hat{\theta}-\theta_0)^\tr\hat{G}(\theta_0-\tilde{\theta})\\
       & +|\hat\alpha|_{1,\kappa}-|\alpha_0|_{1,\kappa}+|\hat\beta|_{1,\pi}-|\beta_0|_{1,\pi}+\lambda_{2}(|\hat{\alpha}|^2-|\alpha_0|^2)+\gamma_{2}(|\hat\beta|^2-|\beta_0|^2)\\
       \geq & \langle \hat G, (\hat\theta-\theta_0)^{\otimes 2}\rangle+2(\hat{\theta}-\theta_0)^\tr\hat{G}(\theta_0-\tilde{\theta})\\
       & +\sum_{j=1}^{\p^0} \kappa_j (|\hat{\alpha}_j|-|\alpha_{0,j}|)+\sum_{h=1}^{\q^0}\pi_h(|\hat{\beta}_h|-|\beta_{0,h}|)+\lambda_2\sum_{j=1}^{\p^0}(|\hat{\alpha}_j|^2-|\alpha_{0,j}|^2)+\gamma_2\sum_{h=1}^{\q^0}(|\hat{\beta}_h|^2-|\beta_{0,h}|^2)\\
    \end{align*}

From the well-known inequality $||x|-|y||\leq |x-y|,$ immediately it follows
\begin{align*}
    \sum_{j=1}^{\p^0} \kappa_j (|\hat{\alpha}_j|-|\alpha_{0,j}|) & \geq - \p^0 a_n |\hat{\alpha}-\alpha_0|,\\
    \sum_{h=1}^{\q^0}\pi_h(|\hat{\beta}_h|-|\beta_{0,h}|) & \geq - \q^0 c_n |\hat{\beta}-\beta_0|.
\end{align*} 

Given a convex function $f$, the following inequalities hold: $$f(y)-f(x)\geq \nabla f(x)^\tr (y-x) \Rightarrow |y|^2 -|x|^2\geq 2 x^\tr(y-x),\quad x,y\in\mathbb R^d.$$ By applying this last inequality in the penalty part related to the Ridge component and the Cauchy-Schwarz inequality $|x^\tr y|\leq |x||y|$, we have:\begin{align*}
   \sum_{j=1}^{\p^0}(|\hat{\alpha}_j|^2-|\alpha_{0,j}|^2) & \geq -2\p^0 |\alpha_0||\hat{\alpha}-\alpha_0|\\
    \sum_{h=1}^{\q^0}(|\hat{\beta}_j|^2-|\beta_{0,j}|^2) & \geq -2\q^0 |\beta_0||\hat{\beta}-\beta_0|\\
\end{align*}
Therefore, we get: \begin{align*}
    0 \geq & \langle \hat G, (\hat\theta-\theta_0)^{\otimes 2}\rangle+2(\hat{\theta}-\theta_0)^\tr\hat{G}(\theta_0-\tilde{\theta})\\
     & - \p^0 \frac{a_n}{\sqrt{n\Delta_n}} |\sqrt{n\Delta_n}(\hat{\alpha}-\alpha_0)|-\q^0 \frac{c_n}{\sqrt{n}} |\sqrt{n}(\hat{\beta}-\beta_0)|\\
    & -2\frac{\lambda_2}{\sqrt{n\Delta_n}}  \p^0 |\alpha_0||\sqrt{n\Delta_n}(\hat{\alpha}-\alpha_0)|-2\frac{\gamma_2}{\sqrt{n}} \q^0 |\beta_0||\sqrt{n}(\hat{\beta}-\beta_0)|\\
    = &\langle \hat D, (A^{-1}(\hat\theta-\theta_0))^{\otimes 2}\rangle+2[A^{-1}(\hat{\theta}-\theta_0)]^\tr\hat{D}[A^{-1}(\theta_0-\tilde{\theta})]\\
    & -\left(\p^0 \frac{a_n}{\sqrt{n\Delta_n}} + \q^0 \frac{c_n}{\sqrt{n}} +2\p^0\frac{\lambda_2}{\sqrt{n\Delta_n}}|\alpha_0| + 2\q^0\frac{\gamma_2}{\sqrt{n}}  |\beta_0|\right)|A^{-1}(\hat{\theta}-\theta_0)|
\end{align*}
and by using the same arguments in the proof of Theorem 1 in \cite{de2021regularized}, one has
\begin{equation}\label{consistency}
    |A^{-1}(\hat{\theta}-\theta_0)| \leq \lVert \hat{D}^{-1}\rVert \left(2 \rVert \hat{D}  \rVert |A^{-1}(\tilde{\theta}-\theta_0)| + \p^0 \frac{a_n}{\sqrt{n\Delta_n}} + \q^0 \frac{c_n}{\sqrt{n}} + 2\p^0 \frac{\lambda_2}{\sqrt{n\Delta_n}} |\alpha_0| + 2\q^0\frac{\gamma_2}{\sqrt{n}} |\beta_0| \right)
\end{equation}
Given that the right-hand side of the inequality is $O_p(1)$ according to the statement, we obtain $A^{-1}\left(\hat{\theta}-\theta_0\right)=O_p(1)$.

(ii) In order to prove the selection consistency, 
   we observe that $\hat{G}=\begin{pmatrix}
        \hat{G}_{\alpha}\\
        \hat{G}_{\beta}
    \end{pmatrix}$ where $\hat{G}_{\alpha}$ and $\hat{G}_{\beta}$ are random matrices of dimension $\p\times\m$ and  $\q\times\m$, respectively. 
    
     Let us assume $\hat\alpha\notin \partial \Theta_\alpha$ and $\hat{\alpha}_{j}\neq 0$  for some $j=\p^0+1,\ldots,\p$. From the KKT condition \begin{equation}
         \frac{1}{\sqrt{n\Delta_n}} \frac{\partial}{\partial \alpha_j}\mathcal{F}(\theta,\tilde\theta)\at[\Bigg]{\theta=\hat{\theta}}=0
     \end{equation} we derive  \begin{align*}
         \left|\frac{2}{\sqrt{n\Delta_n}} \hat{G}_\alpha(j) AA^{-1} (\hat{\theta}-\tilde{\theta})\right|& = \left|\frac{\kappa_{j}}{\sqrt{n\Delta_n}} \mathrm{sgn}(\hat{\alpha}_{j})+2\frac{\lambda_2}{\sqrt{n\Delta_n}}\hat{\alpha}_{j}\right|\\
        % & \geq \left|\frac{\kappa_{j}}{\sqrt{n\Delta_n}}\right|- \frac{\lambda_2}{\sqrt{n\Delta_n}}|\alpha_{j}|\\
         & \geq \left|\frac{b_n}{\sqrt{n\Delta_n}}\right|- \frac{\lambda_2}{\sqrt{n\Delta_n}}|\alpha_{j}|
     \end{align*}
     for some $j=\p^0+1,\ldots,\p$, with $\hat{G}_{\alpha}(j)$ the $j$-th row of $\hat{G}_{\alpha}$.
     
By consistency and the assumptions, we have that $\left|\frac{2}{\sqrt{n\Delta_n}} \hat{G}_\alpha(j) A A^{-1} (\hat{\theta}-\tilde{\theta})\right|=O_p(1)$ and  $\frac{\lambda_2}{\sqrt{n\Delta_n}}|\hat{\alpha}_{j}|=O_p(1)$, while $\frac{b_n}{\sqrt{n\Delta_n}}\overset{p}{\longrightarrow}\infty$. Therefore, for any $j=\p^0+1,\ldots,\p$:
     \begin{equation}
         P\left(\hat{\alpha}_{j}\neq 0,\hat\alpha\notin \partial \Theta_\alpha\right)\leq P\left(\left|2 \frac{1}{\sqrt{n\Delta_n}} \hat{G}_\alpha(j) A_n A_n^{-1} (\hat{\theta}-\tilde{\theta})\right|+\frac{\lambda_2}{\sqrt{n\Delta_n}}|\hat{\alpha}_{j}|\geq \frac{b_n}{\sqrt{n\Delta_n}}\right) \longrightarrow 0
     \end{equation} as $n\longrightarrow\infty$. By means of the same arguments in \cite{suzuki2020penalized}, we get $P(\hat\alpha\in \partial \Theta_\alpha)\longrightarrow 0.$ Therefore
     $$P\left(\hat{\alpha}\neq 0\right)\leq P(\hat\alpha\in \partial \Theta_\alpha)+\sum_{j=\p^0+1}^\p  P\left(\hat{\alpha}_{j}\neq 0,\hat\alpha\notin \partial \Theta_\alpha\right)\longrightarrow 0.$$

Similar steps lead to $P(\hat \beta\neq 0)\longrightarrow 0.$

(iii)
     It is possible to adapt to our case  the proof of Theorem 3 in \cite{de2021regularized}. We have two groups of parameters and two more terms arising from
ridge regularization.  It is useful to rewrite $\hat{G}$ as a block matrix partitioned as follows: $$\hat{G}=\begin{pmatrix}
        \hat{G}^{^{\alpha\alpha}} & \hat{G}^{^{\alpha\beta}}\\
        \hat{G}^{^{\beta\alpha}} & \hat{G}^{^{\beta\beta}}\\
    \end{pmatrix}$$ where $$ \hat{G}^{^{\alpha\alpha}}:=  \begin{pmatrix}
        \hat{G}^{^{\alpha\alpha}}_{\star\star} & \hat{G}^{^{\alpha\alpha}}_{\star\bullet} \\
        \hat{G}^{^{\alpha\alpha}}_{\bullet\star} & \hat{G}^{^{\alpha\alpha}}_{\bullet\bullet}
    \end{pmatrix},\, \hat{G}^{^{\alpha\beta}}:= \begin{pmatrix}
        \hat{G}^{^{\alpha\beta}}_{\star\star} & \hat{G}^{^{\alpha\beta}}_{\star\bullet} \\
        \hat{G}^{^{\alpha\beta}}_{\bullet\star} & \hat{G}^{^{\alpha\beta}}_{\bullet\bullet}
    \end{pmatrix},\, \hat{G}^{^{\beta\alpha}}:=  \begin{pmatrix}
        \hat{G}^{^{\beta\alpha}}_{\star\star} & \hat{G}^{^{\beta\alpha}}_{\star\bullet} \\
        \hat{G}^{^{\beta\alpha}}_{\bullet\star} & \hat{G}^{^{\beta\alpha}}_{\bullet\bullet}
    \end{pmatrix},\, \hat{G}^{^{\alpha\beta}}= \begin{pmatrix}
        \hat{G}^{^{\beta\beta}}_{\star\star} & \hat{G}^{^{\beta\beta}}_{\star\bullet} \\
        \hat{G}^{^{\beta\beta}}_{\bullet\star} & \hat{G}^{^{\beta\beta}}_{\bullet\bullet}
    \end{pmatrix} $$
    
We can rewrite the objective function: 
    \begin{align*}
       \mathcal{F}_n(\theta)= & \ (\theta-\tilde{\theta})^\tr \hat{G}(\theta-\tilde{\theta})+\sum_{j=1}^{\p}\kappa_{j}|\alpha_j|+\sum_{h=1}^{\q} \pi_{h} |\beta_h|+\lambda_{n}\sum_{j=1}^{\p}|\alpha_j|^2+\gamma_{n}\sum_{h=1}^{\q}|\beta_h|^2\\
        = & \  (\alpha-\tilde{\alpha})^\tr_\star \hat{G}^{^{\alpha\alpha}}_{\star\star} (\alpha-\tilde{\alpha})_\star+(\beta-\tilde{\beta})^\tr_\star \hat{G}^{^{\beta\beta}}_{\star\star} (\beta-\tilde{\beta})_\star\\
       & + (\alpha-\tilde{\alpha})^\tr_\bullet \hat{G}^{^{\alpha\alpha}}_{\bullet\bullet} (\alpha-\tilde{\alpha})_\bullet+(\beta-\tilde{\beta})^\tr_\bullet \hat{G}^{^{\beta\beta}}_{\bullet\bullet} (\beta-\tilde{\beta})_\bullet\\
       & + 2 \left[(\alpha-\tilde{\alpha})^\tr_\star \hat{G}^{^{\alpha\alpha}}_{\star\bullet} (\alpha-\tilde{\alpha})_\bullet+ (\beta-\tilde{\beta})^\tr_\star \hat{G}^{^{\beta\beta}}_{\star\bullet} (\beta-\tilde{\beta})_\bullet\right]\\
       & +2\Big[(\alpha-\tilde{\alpha})^\tr_\star\hat{G}^{^{\alpha\beta}}_{\star\star}(\beta-\tilde{\beta})_\star+ (\alpha-\tilde{\alpha})^\tr_\star\hat{G}^{^{\alpha\beta}}_{\star\bullet}(\beta-\tilde{\beta})_\bullet  \\ 
        & + (\alpha-\tilde{\alpha})^\tr_\bullet\hat{G}^{^{\alpha\beta}}_{\bullet\star}(\beta-\tilde{\beta})_\star + (\alpha-\tilde{\alpha})^\tr_\bullet\hat{G}^{^{\alpha\beta}}_{\bullet\bullet}(\beta-\tilde{\beta})_\bullet \Big]\\
        & +\sum_{j=1}^{\p}\kappa_{j}|\alpha_j|+\sum_{h=1}^{\q} \pi_{h} |\beta_h|+\lambda_{n}\sum_{j=1}^{\p}|\alpha_j|^2+\gamma_{n}\sum_{h=1}^{\q}|\beta_h|^2
    \end{align*}
By setting $\check{\theta}:=(\alpha_\star,0,\beta_\star,0)^\tr\in \mathbb{R}^\m$, we have \begin{align*}
   \mathcal{F}_n(\check{\theta}) = & \ (\alpha-\tilde{\alpha})^\tr_\star\hat{G}_{\star\star}^{^{\alpha\alpha}}(\alpha-\tilde{\alpha})_\star + (\beta-\tilde{\beta})^\tr_\star\hat{G}_{\star\star}^{^{\beta\beta}}(\beta-\tilde{\beta})_\star\\
    & + (\tilde{\alpha}_\bullet)^\tr \hat{G}^{^{\alpha\alpha}}_{\bullet\bullet}\tilde{\alpha}_\bullet + (\tilde{\beta}_\bullet)^\tr \hat{G}^{^{\beta\beta}}_{\bullet\bullet}\tilde{\beta}_\bullet\\
    & -2 \left[(\alpha-\tilde{\alpha})^\tr_\star\hat{G}^{^{\alpha\alpha}}_{\star\bullet}\tilde{\alpha}_\bullet + (\beta-\tilde{\beta})^\tr_\star\hat{G}^{^{\beta\beta}}_{\star\bullet}\tilde{\beta}_\bullet \right] \\
    & +2 \left[(\alpha-\tilde{\alpha})^\tr_\star \hat{G}_{\star\star}^{^{\alpha\beta}}(\beta-\tilde{\beta})_\star - (\alpha-\tilde{\alpha})^\tr_\star\hat{G}_{\star\bullet}^{^{\alpha\beta}}\tilde{\beta}_\bullet-(\tilde{\alpha}_\bullet)^\tr \hat{G}_{\bullet\star}^{^{\alpha\beta}}(\beta-\tilde{\beta})_\star + (\tilde{\alpha}_\bullet)^\tr \hat{G}^{^{\alpha\beta}}_{\bullet\bullet} \tilde{\beta}_\bullet \right]\\
    & + \sum_{j=1}^{\p^0}\kappa_{j}|\alpha_j|+\sum_{h=1}^{\q^0} \pi_{h} |\beta_h|+\lambda_{n}\sum_{j=1}^{\p^0}|\alpha_j|^2+\gamma_{n}\sum_{h=1}^{\q^0}|\beta_h|^2
\end{align*}
Let us consider \begin{align*}
    B^\alpha_n & := \left\{\displaystyle \min_{1\leq j\leq \p^0}|\hat{\alpha}_j|>0,\, \hat{\alpha}_\bullet = 0, \,\det(\hat{G}_{\star\star}^{^{\alpha\alpha}}>0) \right\} \\
        B^\beta_n & := \left\{\displaystyle \min_{1\leq h\leq \q^0}|\hat{\beta}_j|>0,\, \hat{\beta}_\bullet = 0, \,\det(\hat{G}_{\star\star}^{^{\beta\beta}}>0) \right\} \\
\end{align*}
By Theorem 1 and 2, $P(B^\alpha_n\cap B^\beta_n )\longrightarrow 1$. If $B^\alpha_n\cap B^\beta_n$ holds, then $\mathfrak{F}_n (\hat{\theta})=\displaystyle\min_{\check{\theta}\in \mathbb{R}^\m_0}\mathfrak{F}_n(\check{\theta})$, where $\mathbb{R}^\m_0:=\left\{\theta\in\mathbb{R}^\m: \alpha_\bullet=0, \beta_\bullet=0 \right\}$. 

Then, on $B^\alpha_n$:
\begin{equation*}
    0=\frac{1}{2}\frac{\partial}{\partial \alpha_\star}\mathcal{F}_n(\theta)\at[\Bigg]{\theta=\hat{\theta}_n}=\hat{G}^{^{\alpha\alpha}}_{\star\star}(\hat{\alpha}-\tilde{\alpha})_\star-\hat{G}^{^{\alpha\alpha}}_{\star\bullet}\tilde{\alpha}_\bullet+\hat{G}^{^{\alpha\beta}}_{\star\star}(\hat{\beta}-\tilde{\beta})_\star-\hat{G}^{^{\alpha\beta}}_{\star\bullet}\tilde{\beta}_\bullet+Z(\hat{\alpha})+U(\hat{\alpha})
\end{equation*} where $Z(\hat{\alpha})=\left(\frac{1}{2}\kappa_1 \mathrm{sgn}(\hat{\alpha}_1), \ldots, \frac{1}{2}\kappa_{\p^0} \mathrm{sgn}(\hat{\alpha}_{\p^0})\right)^\tr$ and $U(\hat{\alpha})=\left(\lambda_n \hat{\alpha}_1, \ldots, \lambda_n\hat{\alpha}_{\p^0}\right)^\tr$. Then, adding and subtracting the true value $\alpha_0$: \begin{equation*}
    (\hat{\alpha}-\alpha_0)_\star = (\tilde{\alpha}-\alpha_0)_\star+(\hat{G}^{^{\alpha\alpha}}_{\star\star})^{-1}\hat{G}^{^{\alpha\alpha}}_{\star\bullet}\tilde{\alpha}_\bullet-(\hat{G}^{^{\alpha\alpha}}_{\star\star})^{-1}\hat{G}^{^{\alpha\beta}}_{\star\star}(\hat{\beta}-\tilde{\beta})_\star+(\hat{G}^{^{\alpha\alpha}}_{\star\star})^{-1}\hat{G}^{^{\alpha\beta}}_{\star\bullet}\tilde{\beta}_\bullet-(\hat{G}^{^{\alpha\alpha}}_{\star\star})^{-1}Z(\hat{\alpha})-(\hat{G}^{^{\alpha\alpha}}_{\star\star})^{-1}U(\hat{\alpha})
\end{equation*}

Let us introduce a $\m^0\times\m$ matrix $$J:=\begin{pmatrix}
    J_\alpha & 0\\
    0 & J_\beta
\end{pmatrix},$$ where $J_\alpha := (\mathbf{I}_{\p^0}(\Gamma^{^{\alpha\alpha}}_{\star\star})^{-1} \Gamma^{^{\alpha\alpha}}_{\star\bullet} )$ and $J_\beta := (\mathbf{I}_{\q^0}(\Gamma^{^{\beta\beta}}_{\star\star})^{-1} \Gamma^{^{\beta\beta}}_{\star\bullet} )$. 

Let $\hat{J}_\alpha := \left(\mathbf{I}_{\p^0}n\Delta_n(\hat{G}^{^{\alpha\alpha}}_{\star\star})^{-1} \frac{1}{n\Delta_n}\hat{G}^{^{\alpha\alpha}}_{\star\bullet} \right)\overset{p}{\longrightarrow}J_\alpha$. Hence multiplying by $\sqrt{n\Delta_n}$ and adding and subtracting $J_\alpha\left\{\sqrt{n\Delta_n}(\tilde{\alpha}-\alpha_0)\right\}$:

\ \\
$\sqrt{n\Delta_n}(\hat{\alpha}-\alpha_0)_\star-J_\alpha\left\{\sqrt{n\Delta_n}(\tilde{\alpha}-\alpha_0)\right\}$\begin{align*}
     = & \ \bbone_{B^\alpha_n}\Big\{\sqrt{n\Delta_n}(\tilde{\alpha}-\alpha_0)_\star+ \sqrt{n\Delta_n}(\hat{G}^{^{\alpha\alpha}}_{\star\star})^{-1}\hat{G}^{^{\alpha\alpha}}_{\star\bullet}\tilde{\alpha}_\bullet - \sqrt{n\Delta_n}(\hat{G}^{^{\alpha\alpha}}_{\star\star})^{-1}\hat{G}^{^{\alpha\beta}}_{\star\star}(\hat{\beta}-\tilde{\beta})_\star\\
    & + \sqrt{n\Delta_n} (\hat{G}^{^{\alpha\alpha}}_{\star\star})^{-1}\hat{G}^{^{\alpha\beta}}_{\star\bullet}\tilde{\beta}_\bullet- \sqrt{n\Delta_n}(\hat{G}^{^{\alpha\alpha}}_{\star\star})^{-1}Z(\hat{\alpha})- \sqrt{n\Delta_n}(\hat{G}^{^{\alpha\alpha}}_{\star\star})^{-1}U(\hat{\alpha})-J_\alpha\left\{\sqrt{n\Delta_n}(\tilde{\alpha}-\alpha_0)\right\}\Big\}\\
    & + \bbone_{B^\alpha_n}\Big\{\sqrt{n\Delta_n}(\hat{\alpha}-\alpha_0)_\star-J_\alpha\left\{\sqrt{n\Delta_n}(\tilde{\alpha}-\alpha_0)\right\}\Big\}\\
    = & \ \bbone_{B^\alpha_n}\Big\{(\hat{J}_\alpha-J_\alpha)\left\{\sqrt{n\Delta_n}(\tilde{\alpha}-\alpha_0)\right\}-\sqrt{n\Delta_n}(\hat{G}^{^{\alpha\alpha}}_{\star\star})^{-1}\hat{G}^{^{\alpha\beta}}_{\star\star}(\hat{\beta}-\tilde{\beta})_\star+\sqrt{n\Delta_n} (\hat{G}^{^{\alpha\alpha}}_{\star\star})^{-1}\hat{G}^{^{\alpha\beta}}_{\star\bullet}\tilde{\beta}_\bullet\\
    & - \sqrt{n\Delta_n}(\hat{G}^{^{\alpha\alpha}}_{\star\star})^{-1}Z(\hat{\alpha})- \sqrt{n\Delta_n}(\hat{G}^{^{\alpha\alpha}}_{\star\star})^{-1}U(\hat{\alpha})\Big\}+ \bbone_{B^\alpha_n}\Big\{\sqrt{n\Delta_n}(\hat{\alpha}-\alpha_0)_\star-J_\alpha\left\{\sqrt{n\Delta_n}(\tilde{\alpha}-\alpha_0)\right\}\Big\}\\
    = & \ \bbone_{B^\alpha_n}\Big\{(\hat{J}_\alpha-J_\alpha) \left\{\sqrt{n\Delta_n}(\tilde{\alpha}-\alpha_0)\right\}+o_p(1)\Big\}+\bbone_{B^\alpha_n}\Big\{\sqrt{n\Delta_n}(\hat{\alpha}-\alpha_0)_\star-J_\alpha\left\{\sqrt{n\Delta_n}(\tilde{\alpha}-\alpha_0)\right\}\Big\}
\end{align*} 

where the last step holds because: \begin{equation*}
    n\Delta_n (\hat{G}_{\star\star}^{^{\alpha\alpha}})^{-1}\frac{1}{\sqrt{n\Delta_n}}\frac{1}{\sqrt{n}}\hat{G}^{^{\alpha\beta}}_{\star\star}\sqrt{n}(\hat{\beta}-\tilde{\beta})_\star\bbone_{B^{\alpha}_n}=o_p(1)O_p(1)=o_p(1)
\end{equation*}\begin{equation*}
    n\Delta_n(\hat{G}^{^{\alpha\alpha}}_{\star\star})^{-1} \frac{1}{\sqrt{n\Delta_n}}\frac{1}{\sqrt{n}}\hat{G}^{^{\alpha\beta}}_{\star\bullet}\sqrt{n}\tilde{\beta}_\bullet\bbone_{B^\alpha_n}=o_p(1)
\end{equation*}\begin{equation*}
    n\Delta_n(\hat{G}^{^{\alpha\alpha}}_{\star\star})^{-1}\frac{1}{\sqrt{n\Delta_n}}Z(\hat{\alpha})\bbone_{B^\alpha_n}=o_p(1)
\end{equation*}
\begin{equation*}
    n\Delta_n(\hat{G}^{^{\alpha\alpha}}_{\star\star})^{-1}\frac{1}{\sqrt{n\Delta_n}}U(\hat{\alpha})\bbone_{B^\alpha_n}=o_p(1)
\end{equation*}

Finally,  $$\sqrt{n\Delta_n}(\hat{\alpha}-\alpha_0)_\star-J_\alpha\Big\{\sqrt{n\Delta_n}(\tilde{\alpha}-\alpha_0)\Big\}\overset{p}{\longrightarrow}0$$

The same steps can be done considering on $B^\beta_n$, with different rates and the result holds.

Adding the property P3 and following the blockwise inversion of $\Gamma^{^{\alpha\alpha}}$ and $\Gamma^{^{\beta\beta}}$, exploited in the proof of Theorem 3 in \cite{de2021regularized}, we get $$\left(\sqrt{n\Delta_n}(\hat{\alpha}-\alpha_0)_\star,\, \sqrt{n}(\hat{\beta}-\beta_0)_\star\right)^\tr\overset{d}{\longrightarrow}N_{\m^0}\left(0, \mathrm{diag}\left((\Gamma_{\star\star}^{^{\alpha\alpha}})^{-1},(\Gamma_{\star\star}^{^{\beta\beta}})^{-1}\right)\right)$$

%Nevertheless, these latter can be easily controlled under our conditions. For these reasons we omit the details  of the proof.  
\end{proof}

\begin{proof}[Proof of Theorem \ref{thm:bound}]
    Starting from \eqref{consistency}, by definition of convexity we have the following inequalities, for any $r\geq 1$:

    \begin{align*}
        |A_n^{-1}(\hat{\theta}_n-\theta_0)|^r  \leq & \ \bigg | 2 \lVert \hat{D}_{n}^{-1}\rVert \lVert \hat{D}_n \rVert |A^{-1}_n(\tilde{\theta}_n-\theta_0)|\\
         & + \lVert \hat{D}_{n}^{-1}\rVert \left(\p^0 \frac{a_n}{\sqrt{n\Delta_n}} + \q^0 \frac{c_n}{\sqrt{n}} + 2\p^0 \frac{\lambda_n}{\sqrt{n\Delta_n}} |\alpha_0| + 2\q^0\frac{\gamma_n}{\sqrt{n}} |\beta_0|\right)\bigg|^r\\
         \leq & \ \bigg| \frac{1}{2}\Big|2 \cdot 2 \lVert \hat{D}_{n}^{-1}\rVert \lVert \hat{D}_n \rVert |A^{-1}_n(\tilde{\theta}_n-\theta_0)| \Big|\\
         & + \frac{1}{2}\Big|2\cdot  \lVert \hat{D}_{n}^{-1}\rVert \left(\p^0 \frac{a_n}{\sqrt{n\Delta_n}} + \q^0 \frac{c_n}{\sqrt{n}} + 2\p^0 \frac{\lambda_n}{\sqrt{n\Delta_n}} |\alpha_0| + 2\q^0\frac{\gamma_n}{\sqrt{n}} |\beta_0|\right)\Big|\bigg|^r\\
         \leq & \  \frac{1}{2}\Big|2 \cdot 2 \lVert \hat{D}_{n}^{-1}\rVert \lVert \hat{D}_n \rVert |A^{-1}_n(\tilde{\theta}_n-\theta_0)| \Big|^r\\
         & + \frac{1}{2}\Big|2\cdot  \lVert \hat{D}_{n}^{-1}\rVert \left(\p^0 \frac{a_n}{\sqrt{n\Delta_n}} + \q^0 \frac{c_n}{\sqrt{n}} + 2\p^0 \frac{\lambda_n}{\sqrt{n\Delta_n}} |\alpha_0| + 2\q^0\frac{\gamma_n}{\sqrt{n}} |\beta_0|\right)\Big|^r\\
         = & \ 2^{2r-1} \lVert \hat{D}_{n}^{-1}\rVert^r \lVert \hat{D}_n \rVert^r |A^{-1}_n(\tilde{\theta}_n-\theta_0)|^r\\
         & + 2^{r-1}  \lVert \hat{D}_{n}^{-1}\rVert^r \left|\p^0 \frac{a_n}{\sqrt{n\Delta_n}} + \q^0 \frac{c_n}{\sqrt{n}} + 2\p^0 \frac{\lambda_n}{\sqrt{n\Delta_n}} |\alpha_0| + 2\q^0\frac{\gamma_n}{\sqrt{n}} |\beta_0|\right|^r
    \end{align*}
    Applying the same previous steps, from the second part of the inequality it follows:

    \ \\ 
    $\left|\p^0 \frac{a_n}{\sqrt{n\Delta_n}} + \q^0 \frac{c_n}{\sqrt{n}} + 2\p^0 \frac{\lambda_n}{\sqrt{n\Delta_n}} |\alpha_0| + 2\q^0\frac{\gamma_n}{\sqrt{n}} |\beta_0|\right|^r$ \begin{align*}
         \leq & \  2^{r-1}\left|\p^0 \frac{a_n}{\sqrt{n\Delta_n}}\right|^r+ 2^{r-1} \left|\q^0 \frac{c_n}{\sqrt{n}} + 2\p^0 \frac{\lambda_n}{\sqrt{n\Delta_n}} |\alpha_0| + 2\q^0\frac{\gamma_n}{\sqrt{n}} |\beta_0|\right|^r\\
         \leq & \ 2^{r-1} \left|\p^0 \frac{a_n}{\sqrt{n\Delta_n}}\right|^r + 2^{r-1}\left(2^{r-1}\left|\q^0 \frac{c_n}{\sqrt{n}}\right|^r+2^{r-1}\left|2\p^0 \frac{\lambda_n}{\sqrt{n\Delta_n}} |\alpha_0| + 2\q^0\frac{\gamma_n}{\sqrt{n}} |\beta_0|\right|^r\right)\\
         \leq & \ 2^{r-1}\left|\p^0 \frac{a_n}{\sqrt{n\Delta_n}}\right|^r + 2^{2(r-1)} \left|\q^0 \frac{c_n}{\sqrt{n}}\right|^r + 2^{2(r-1)} \left(2^{r-1} \left|2\p^0 \frac{\lambda_n}{\sqrt{n\Delta_n}} |\alpha_0|\right|^r + 2^{r-1}\left| 2\q^0\frac{\gamma_n}{\sqrt{n}} |\beta_0|\right|^r \right)\\
         = & \ 2^{r-1}\left|\p^0 \frac{a_n}{\sqrt{n\Delta_n}}\right|^r + 2^{2(r-1)} \left|\q^0 \frac{c_n}{\sqrt{n}}\right|^r + 2^{4r-3} \left|\p^0 \frac{\lambda_n}{\sqrt{n\Delta_n}} |\alpha_0|\right|^r + 2^{4r-3}\left| \q^0\frac{\gamma_n}{\sqrt{n}} |\beta_0|\right|^r 
    \end{align*} Then 

    \ \\
    $2^{r-1}  \lVert \hat{D}_{n}^{-1}\rVert^r \left|\p^0 \frac{a_n}{\sqrt{n\Delta_n}} + \q^0 \frac{c_n}{\sqrt{n}} + 2\p^0 \frac{\lambda_n}{\sqrt{n\Delta_n}} |\alpha_0| + 2\q^0\frac{\gamma_n}{\sqrt{n}} |\beta_0|\right|^r$ \begin{align*}
        \leq & \ \lVert \hat{D}_{n}^{-1}\rVert^r \left( 2^{2(r-1)}\left|\p^0 \frac{a_n}{\sqrt{n\Delta_n}}\right|^r + 2^{3(r-1)} \left|\q^0 \frac{c_n}{\sqrt{n}}\right|^r + 2^{5r-4} \left|\p^0 \frac{\lambda_n}{\sqrt{n\Delta_n}} |\alpha_0|\right|^r + 2^{5r-4}\left| \q^0\frac{\gamma_n}{\sqrt{n}} |\beta_0|\right|^r \right)
    \end{align*}
Finally from the Cauchy-Schwarz inequality

\begin{align*}
    \mathbb{E}\left[|A_n^{-1}(\hat{\theta}_n-\theta_0)|^r\right]  \leq & \ 2^{2r-1} \sqrt{\mathbb{E} \left[\lvert\hat{D}_n^{-1}\rVert^{2r}\right]}\left(\mathbb{E}\left[\lVert\hat{D}_n\rVert^{4r}\right]\right)^\frac{1}{4}\left(\mathbb{E}\left[|A_n^{-1}(\hat{\theta}_n-\theta_0)|^{4r}\right]\right)^{\frac{1}{4}}\\
    & + \lVert \hat{D}_{n}^{-1}\rVert^r \bigg( 2^{2(r-1)}\mathbb{E}\left[\left|\p^0 \frac{a_n}{\sqrt{n\Delta_n}}\right|^r\right] + 2^{3(r-1)} \mathbb{E}\left[\left|\q^0 \frac{c_n}{\sqrt{n}}\right|^r\right]\\
    & + 2^{5r-4} \mathbb{E}\left[\left|\p^0 \frac{\lambda_{2,n}}{\sqrt{n\Delta_n}} |\alpha_0|\right|^r\right] + 2^{5r-4}\mathbb{E}\left[\left| \q^0\frac{\gamma_{2,n}}{\sqrt{n}} |\beta_0|\right|^r\right] \bigg)
\end{align*}

From the assumptions, the polynomial-type large deviation result (25) and Proposition 1 in \cite{yoshida2011polynomial}, the uniform $L^r$-boundedness of the estimator holds.

\end{proof}

\begin{proof}[Proof of Lemma \ref{lem:non-asy}]
    Let $\tilde \theta_n = (\tilde \alpha_n, \tilde \beta_n)$ be the quasi likelihood estimator. 
     By Taylor expansion, the Cauchy-Schwartz inequality and A5$(r)$-$(i)$ we have that 
            \begin{align*}
                |\ell_n(\tilde \alpha_n, \tilde \beta_n)  - \ell_n(\theta_0)|
                 & \leq 
                \left| \int_0^1 \lla \partial_\alpha \ell_n(\alpha_0 + u(\tilde \alpha_n - \alpha_0), \tilde \beta_n), \tilde \alpha_n - \alpha_0 \rra\mathrm d u \right|
                \\ & \leq 
                \int_0^1 
                |\partial_\alpha \overline{\ell_n}(\alpha_0 + u(\tilde \tilde \alpha_n - \alpha_0), \tilde \beta_n )| \mathrm d u 
                \, |\sqrt{n \Delta_n}(\tilde \alpha_n - \alpha_0)|
                \\
                & \leq 
                \xi_n \sqrt {\p} |\sqrt{n \Delta_n}(\tilde \alpha_n - \alpha_0)|.
            \end{align*}
            Similarly, under the assumption that $|\tilde \theta_n - \theta_0| \leq r$ and by A5$(r)-(ii)$ we have that
           \begin{align*}
                |\ell_n(\theta_0) - \ell_n(\tilde \alpha_n, \tilde \beta_n)|
                 & = 
                 \left| \int_0^1 (1-u) 
                 \lla \partial^2_{\alpha \alpha} \ell_n(\tilde \alpha_n + u(\alpha_0 - \tilde \alpha_n), \tilde \beta_n ) , (\tilde \alpha_n - \alpha_0)^{\otimes 2} \rra \mathrm d u \right|     
                 \\ & \geq
                 \frac{\mu}{2} |\sqrt{n \Delta_n} (\tilde \alpha_n - \alpha_0)|^2.
            \end{align*}
           By combining the two inequalities above we get the first claim. 
            The proof for the second inequality is analogous.
\end{proof}

\begin{proof}[Proof of Theorem \ref{thm:ineq}]

By setting $B = \hat G_n^{1/2}$, $y  = \hat G_n^{1/2} \tilde \theta_n$, $\epsilon = \hat G_n^{1/2} (\tilde \theta_n - \theta_0)$
the adaptive Elastic-Net estimator can be 
rewritten as the solution of a linear regression-type minimization problem; i.e.
\[
\hat \theta_n = \hat \theta_n(\lambda_1,\gamma_1, \lambda_2,\gamma_2)
= \arg \min_{\theta}\left\{
| y - B \theta |^2 + \lambda_1 |\alpha|_{1, \mathbf {\bf\kappa}_n} +\gamma_1 |\beta|_{1, \mathbf {\bf\pi}_n}+\theta^{\tr} C(\lambda_2,\gamma_2)\theta\right\}\]
where $C(\lambda_2,\gamma_2):=\text{diag}(\lambda_2 {\bf I}_{\p},\gamma_2{\bf I}_{\q}).$%  $a(\lambda_2,\gamma_2)=(\lambda_2,...,\lambda_2,\gamma_2,...,\gamma_2)^{\tr}$ is a $\p+\q$ vector.

Then the Ridge estimator $(\lambda_1=\gamma_1=0)$ is given by
\[
\hat \theta_n( \lambda_2,\gamma_2) = 
(\hat B^\top B + C(\lambda_2,\gamma_2))^{-1} B^\top y = (\hat G_n + C(\lambda_2,\gamma_2))^{-1} \hat G_n \tilde \theta_n=(\hat \alpha_n( \lambda_2),\hat \beta_n( \gamma_2))^\tr
\]
where $\hat \alpha_n( \lambda_2)$ and $\hat \beta_n( \gamma_2)$ are the Ridge estimators related to the sub-optimization problems $\mathcal F_{1,n}$ and $\mathcal F_{2,n},$ respectively; i.e.
$$\hat \alpha_n( \lambda_2)=(\hat G_n^{\alpha\alpha} + \lambda_2 {\bf I}_{\p})^{-1} \hat G_n^{\alpha\alpha} \tilde \alpha_n,$$
$$\hat \beta_n( \gamma_2)=(\hat G_n^{\beta\beta} + \gamma_2 {\bf I}_{\q})^{-1} \hat G_n^{\beta\beta} \tilde \beta_n.$$
    By following the same steps as in \cite{zou2009adaptive}
    we have that 
    \begin{align*}
        |\hat \theta_n - \theta_0|^2
        &\leq 
        2 |\hat \theta_n - \hat \theta_n(\lambda_2,\gamma_2) |^2
        + 
        2  |\hat \theta_n(\lambda_2,\gamma_2) - \theta_0 |^2\\
        &=2|\hat \alpha_n - \hat \alpha_n(\lambda_2) |^2
        + 
        2  |\hat \alpha_n(\lambda_2)- \alpha_0 |^2\\
        &\quad +2|\hat \beta_n - \hat \beta_n(\gamma_2) |^2
        + 
        2  |\hat \beta_n(\gamma_2))- \beta_0 |^2
    \end{align*}
    Since
    $$\hat \alpha_n(\lambda_2)- \alpha_0 =(\hat G_n^{\alpha\alpha} + \lambda_2 {\bf I}_{\p})^{-1}\hat G_n^{\alpha\alpha}(\tilde\alpha_n-\alpha_0)-(\hat G_n^{\alpha\alpha} + \lambda_2 {\bf I}_{\p})^{-1} \lambda_2\alpha_0$$
    we have that
       \begin{align}\label{eq:proofbound1}
        |\hat \alpha_n(\lambda_2)- \alpha_0  |^2&\leq 2|(\hat G_n^{\alpha\alpha} + \lambda_2 {\bf I}_{\p})^{-1} \lambda_2\alpha_0|^2+2|(\hat G_n^{\alpha\alpha} + \lambda_2 {\bf I}_{\p})^{-1}\hat G_n^{\alpha\alpha}(\tilde\alpha_n-\alpha_0)|^2\notag \\
         &\leq \frac{2}{(\tau_{\min}(\hat G_n^{\alpha\alpha})+\lambda_2)^2}\left(\lambda_2^2|\alpha_0|^2+  | \hat G_n^{\alpha\alpha}(\tilde\alpha_n-\alpha_0) |^2\right)\notag\\
         %&\leq \frac{2}{(\tau_{\min}(\hat G_n^{\alpha\alpha})+\lambda_2)^2}\left(\lambda_2^2|\alpha_0|^2+  || \hat G_n^{\alpha\alpha}||^2\,|\tilde\alpha_n-\alpha_0 |^2\right)\notag\\
        % &\leq \frac{2}{(\tau_{\min}(\hat G_n)+\min\{\lambda_2,\gamma_2\})^2}\left(\max\{\lambda_2,\gamma_2\}^2|\theta_0|^2+ \text{tr} (\hat G_n^2)\text{tr} ((\tilde\theta_n-\theta_0)) ^{\otimes 2})\right)\\
          &\leq \frac{2}{(\tau_{\min}(\hat G_n^{\alpha\alpha})+\lambda_2)^2}\left(\lambda_2^2|\alpha_0|^2+  \tau_{\max}((\hat G_n^{\alpha\alpha})^2)|   \tilde\alpha_n-\alpha_0| ^{ 2}\right),
    \end{align}
    %or equivalently
     %      \begin{align}\label{eq:proofbound2}
      %  |\hat \theta_n(\lambda_2,\gamma_2)-\theta_0 |
      %        \leq \frac{\max\{\lambda_2,\gamma_2\}|\theta_0|+ \sqrt{\m \tau_{\max}(\hat G_n^2)}|\tilde\theta_n-\theta_0| }{\tau_{\min}(\hat G_n)+\min\{\lambda_2,\gamma_2\}}.
    %\end{align}
    Furthermore

    \begin{align*}
     (\tau_{\min}(\hat G_n^{\alpha\alpha})+\lambda_2)|\hat\alpha_n-\hat\alpha_n(\lambda_2)|^2&\leq  (\hat\alpha_n-\hat\alpha(\lambda_2))^\tr  (\hat G_n^{\alpha\alpha} + \lambda_2{\bf I}_{\p} ) (\hat\alpha_n-\hat\alpha_n(\lambda_2))\\
     &\leq \lambda_1|\kappa_n||\hat\alpha_n-\hat\alpha_n(\lambda_2)|
    \end{align*}
    and then
    \begin{align}\label{eq:proofbound3}
        |\hat\alpha_n-\hat\alpha_n(\lambda_2) |
        \leq \frac{\lambda_1 |\mathbf{\kappa}_n|}{\tau_{\min}(\hat G_n^{\alpha\alpha})+\lambda_2}.
    \end{align}

  %  \begin{align}
   %     \|\hat \theta_n(0, \lambda_2) - \theta_0\|^2
   %     \leq 
    %    2 (\tau_{\min}(\hat G_n) + \lambda_2) ^{-2} 
    %    (
     %   \lambda_2^2 \|\theta_0\|^2 + 
      %   \langle \hat G_n^2 , (\tilde \theta_n - \theta_0 )^{\otimes 2}\rangle
      %  )
    %\end{align}
   % Now, we observe that
    %\begin{align}\label{eq:proofbound4}
    %\frac{\sqrt{\m\tau_{\max}(\hat G_n^2)}|\tilde\theta_n-\theta_0|}{\tau_{\min}(\hat G_n)+\min\{\lambda_2,\gamma_2\}}&\leq  \frac{\sqrt{\m\tau_{\max}(\hat G_n^2)}|\tilde\theta_n-\theta_0|}{\tau_{\min}(\hat D_n)\tau_{\min}(A_n^{-1})^2+\min\{\lambda_2,\gamma_2\}}\notag\\
    %&\leq\frac{\sqrt{\m\tau_{\max}(\hat G_n^2)}||A_n|| \,|A_n^{-1}(\tilde\theta_n-\theta_0)|}{\tau_{\min}(\hat D_n)\tau_{\min}(A_n^{-1})^2+\min\{\lambda_2,\gamma_2\}}\notag\\
     %&\leq\frac{\sqrt{\m^2\tau_{\max}(\hat G_n^2)\tau_{\max}(A_n^2)}\,|A_n^{-1}(\tilde\theta_n-\theta_0)|}{\tau_{\min}(\hat D_n)\tau_{\min}(A_n^{-1})^2+\min\{\lambda_2,\gamma_2\}}\notag\\
     %&=\frac{\m\sqrt{\tau_{\max}(A_n^2)\tau_{\max}(\hat G_n^2)\tau_{\max}(A_n^2)}\,|A_n^{-1}(\tilde\theta_n-\theta_0)|}{\sqrt{\tau_{\max}(A_n^2)}[\tau_{\min}(\hat D_n)\tau_{\min}(A_n^{-1})^2+\min\{\lambda_2,\gamma_2\}]}\notag\\
      %&\leq\frac{\m\sqrt{\tau_{\max}(\hat D_n^2)}\,|A_n^{-1}(\tilde\theta_n-\theta_0)|}{\sqrt{\tau_{\max}(A_n^2)}\tau_{\min}(\hat D_n)\tau_{\min}(A_n^{-1})^2+\sqrt{\tau_{\max}(A_n^2)}\min\{\lambda_2,\gamma_2\}}\notag\\
 %     &=\frac{\m\sqrt{\tau_{\max}(\hat D_n^2)}\,|A_n^{-1}(\tilde\theta_n-\theta_0)|}{\max\{\frac{1}{\sqrt{n}},\frac{1}{\sqrt{n\Delta_n}}\}[\tau_{\min}(\hat D_n)\min\{n,n\Delta_n\}+\min\{\lambda_2,\gamma_2\}]}
    %\end{align}

    Therefore, from  \eqref{eq:proofbound1} and \eqref{eq:proofbound3}, we obtain
    \begin{align}\label{eq:alphain0}
        |\hat\alpha_n-\alpha_0|^2\leq 
        \frac{4}{(\tau_{\min}(\hat G_n^{\alpha\alpha})+\lambda_2)^2}\left(\lambda_2^2|\alpha_0|^2+ 
        \tau_{\max}((\hat G_n^{\alpha\alpha})^2)|   \tilde\alpha_n-\alpha_0| ^{ 2}+\lambda_1^2 |\kappa_n|^2\right).
    \end{align}
    For two real symmetric positive definite matrices $A$ and $B$, we have the following inequalities involving the minimum and maximum eigenvalues $\frac{1}{\tau_{\min}(AB)}\leq\frac{1}{\tau_{\min}(A)}\frac{1}{\tau_{\min}(B)} $ and $\tau_{\max}(AB)\leq \tau_{\max}(A)\tau_{\max}(B).$ 
    Therefore from \eqref{eq:alphain0}, we get
     \begin{align}
        |\hat\alpha_n-\alpha_0|^2\leq \frac{4}{(n\Delta_n\tau_{\min}(\hat D_n^{\alpha\alpha})+\lambda_2)^2} \left(\lambda_2^2|\alpha_0|^2+ (n\Delta_n\tau_{\max}(\hat D_n^{\alpha\alpha})|\tilde\alpha_n-\alpha_0|) ^{ 2}+\lambda_1^2 |\kappa_n|^2\right).
    \end{align}

    By similar steps we can write down
        \begin{align*}
        |\hat\beta_n-\beta_0|^2&\leq \frac{4}{(\tau_{\min}(\hat G_n^{\beta\beta})+\gamma_2)^2}\left(\gamma_2^2|\beta_0|^2+ \tau_{\max}((\hat G_n^{\beta\beta})^2)|   \tilde\beta_n-\beta_0| ^{ 2}+\gamma_1^2 |\pi_n|^2\right)
        \\ &\leq
         \frac{4}{(n \tau_{\min}( \hat D_n^{\beta\beta})+\gamma_2)^2}\left(\gamma_2^2|\beta_0|^2+  (n\tau_{\max}(\hat D_n^{\beta\beta})|   \tilde\beta_n-\beta_0|)^{ 2}+\gamma_1^2 |\pi_n|^2\right)
    \end{align*}

   % Now, we prove \eqref{eq:riskbound}. By exploiting \eqref{eq:proofbound1} and \eqref{eq:proofbound3}, we get
   % \begin{align*}
%\mathbb E|\hat \theta_n - \theta_0|^2&=2\mathbb E|\hat\theta_n-\hat \theta_n(\lambda_2,\gamma_2)|^2+2\mathbb E|\hat \theta_n(\lambda_2,\gamma_2)-\theta_0 |^2\\
%&\leq 4\max\{\lambda_2,\gamma_2\}^2|\theta_0|^2\mathbb E \left(\frac{ 1}{ \tau_{\min}(\hat G_n)+\min\{\lambda_2,\gamma_2\}}\right)^2 + 2\mathbb E \left(\frac{\max\{\lambda_1 |\mathbf{\kappa}_n|,\gamma_1|{\bf \pi}_n|\}}{\tau_{\min}(\hat G_n)+\min\{\lambda_2,\gamma_2\}}\right)^2 \\
%&\quad+4 \m\mathbb E \left(  \frac{\tau_{\max}(\hat G_n^2)|\tilde\theta_n-\theta_0| ^{ 2}}{(\tau_{\min}(\hat G_n)+\min\{\lambda_2,\gamma_2\})^2}\right)\\
%&\leq 4\max\{\lambda_2,\gamma_2\}^2|\theta_0|^2\mathbb E \left(\frac{ 1}{ \tau_{\min}(\hat G_n)+\min\{\lambda_2,\gamma_2\}}\right)^2 + 2\mathbb E \left(\frac{\max\{\lambda_1 |\mathbf{\kappa}_n|,\gamma_1|{\bf \pi}_n|\}}{\tau_{\min}(\hat G_n)+\min\{\lambda_2,\gamma_2\}}\right)^2 \\
%&\quad+4 \m\sqrt{\mathbb E \left(  \frac{\tau_{\max}^{1/2}(\hat G_n^2)}{\tau_{\min}(\hat G_n)+\min\{\lambda_2,\gamma_2\}}\right)^4}\sqrt{\mathbb E|\tilde\theta_n-\theta_0| ^{ 4}}
%    \end{align*}
%where in the last step we have used the Cauchy-Schwarz inequality.
\end{proof}

\begin{proof}[Proof of Theorem \ref{thm:ineq-ii}]
We prove the result for $\hat{\alpha}_n$, the proof for $\hat{\beta}_n$   
is analogous. 
From Theorem \ref{thm:ineq}, inequality \eqref{eq:nabound1a} holds true, i.e. 
\begin{align}
        |\hat\alpha_n-\alpha_0|\leq \frac{2}{n\Delta_n\tau_{\min}(\hat D_n^{\alpha\alpha})+\lambda_2}\left(\lambda_2|\alpha_0|+ n\Delta_n\tau_{\max}(\hat D_n^{\alpha\alpha})  |\tilde\alpha_n-\alpha_0|+\lambda_1|\kappa_n|\right).
    \end{align}
On the event $\{|A_n^{-1}(\tilde \theta_n - \theta_0)| \leq r\}$, which implies $\{|\tilde \theta_n - \theta_0| \leq r/n\Delta_n\}$ Lemma \ref{lem:non-asy} with A5$(r/ n\Delta_n)$ gives 
\[
 \sqrt{n \Delta_n} |\tilde\alpha_n - \alpha_0| \leq \frac{2 \xi_n}{\mu} \sqrt{\p}.
\]
By combining the two inequalities and by using A6, inequality \eqref{eq:nabound2a} immediately follows. It remains to show that it holds with the desired probability. This is a consequence of the polynomial large deviation results in \cite{yoshida2011polynomial} (see e.g. formula 2.14 in \cite{yoshida2022nonlin}).

\end{proof}

\begin{proof}[Proof of Theorem \ref{theo:prederr}]
    We have that
    \begin{align*}
	\mathsf{MAE}(\hat X_{T_n + h}) &= \mathbb E|X_{T_n + h} - \hat X_{T_n + h}| %=
	%\mathbb E[X_{T_n + h} - X_{T_n}  -  b(X_{T_n}, \hat \alpha_n )h ]^2 
	\\&=
	 \mathbb E|X_{T_n + h} - X_{T_n}  - h b(X_{T_n}, \alpha_0 ) -  h (b(X_{T_n}, \hat \alpha_n ) - b(X_{T_n}, \alpha_0 ))| 
    \\ &\leq
    \left(
    \mathbb E \left[ \mathbb E [ |X_{T_n + h} - X_{T_n}  - h b(X_{T_n}, \alpha_0) |^2 | \mathcal F_{T_n}] \right]
    \right)^{\frac 12} 
    + 
    h
    \mathbb E[C(X_{T_n}) |\hat{\alpha} - \alpha_0|]
    \\ & \leq 
    \left ( \mathbb E [h \, \mathrm{tr} (\Sigma(X_{t_n}, \beta_0)) + R(h^2, \theta_0, X_{T_n})] 
    \right)^\frac 12
    + h
    \left( \mathbb EC^2(X_{T_n}) \mathbb E|\hat{\alpha}_n - \alpha_0|^{2} \right)^{\frac 12}
    \\&\leq 
    C_1 \sqrt h + C_2 h + D h 
    \left( \mathbb E|\hat{\alpha}_n - \alpha_0|^{2} \right)^{\frac 12},
\end{align*}
where we applied the Cauchy-Schwartz inequality and in the second step we applied Lemma 7 in \cite{kess}, adapted to the multidimensional case. The conclusion follows by \autoref{thm:ineq}. 	
\end{proof}

\bibliographystyle{abbrv}

\bibliography{biblio}

\begin{thebibliography}{10}

\bibitem{amorino2024sampling}
C.~Amorino, F.~Pina, and M.~Podolskij.
\newblock Sampling effects on lasso estimation of drift functions in high-dimensional diffusion processes.
\newblock {\em arXiv preprint arXiv:2408.08638}, 2024.

\bibitem{beck2009fast}
A.~Beck and M.~Teboulle.
\newblock A fast iterative shrinkage-thresholding algorithm for linear inverse problems.
\newblock {\em SIAM journal on imaging sciences}, 2(1):183--202, 2009.

\bibitem{bolte2014proximal}
J.~Bolte, S.~Sabach, and M.~Teboulle.
\newblock Proximal alternating linearized minimization for nonconvex and nonsmooth problems.
\newblock {\em Mathematical Programming}, 146(1):459--494, 2014.

\bibitem{carpi2022}
T.~Carpi, A.~Hino, S.~M. Iacus, and G.~Porro.
\newblock The impact of covid-19 on subjective well-being: Evidence from twitter data.
\newblock {\em Journal of Data Science}, 21(4):761--780, 2022.

\bibitem{ciolek2022lasso}
G.~Ciolek, D.~Marushkevych, and M.~Podolskij.
\newblock On lasso estimator for the drift function in diffusion models.
\newblock {\em arXiv preprint arXiv:2209.05974}, 2022.

\bibitem{10.1214/20-EJS1775}
G.~Ciołek, D.~Marushkevych, and M.~Podolskij.
\newblock {On Dantzig and Lasso estimators of the drift in a high dimensional Ornstein-Uhlenbeck model}.
\newblock {\em Electronic Journal of Statistics}, 14(2):4395 -- 4420, 2020.

\bibitem{de2012adaptive}
A.~De~Gregorio and S.~M. Iacus.
\newblock Adaptive lasso-type estimation for multivariate diffusion processes.
\newblock {\em Econometric Theory}, 28(4):838--860, 2012.

\bibitem{de2021regularized}
A.~De~Gregorio and F.~Iafrate.
\newblock Regularized bridge-type estimation with multiple penalties.
\newblock {\em Annals of the Institute of Statistical Mathematics}, 73(5):921--951, 2021.

\bibitem{10.3150/22-BEJ1574}
N.~Dexheimer and C.~Strauch.
\newblock {On Lasso and Slope drift estimators for Lévy-driven Ornstein–Uhlenbeck processes}.
\newblock {\em Bernoulli}, 30(1):88 -- 116, 2024.

\bibitem{fan2001variable}
J.~Fan and R.~Li.
\newblock Variable selection via nonconcave penalized likelihood and its oracle properties.
\newblock {\em Journal of the American statistical Association}, 96(456):1348--1360, 2001.

\bibitem{fan2006statistical}
J.~Fan and R.~Li.
\newblock Statistical challenges with high dimensionality: Feature selection in knowledge discovery.
\newblock {\em arXiv preprint math/0602133}, 2006.

\bibitem{fan2004nonconcave}
J.~Fan and H.~Peng.
\newblock {Nonconcave penalized likelihood with a diverging number of parameters}.
\newblock {\em The Annals of Statistics}, 32(3):928 -- 961, 2004.

\bibitem{florens1989approximate}
D.~Florens-Zmirou.
\newblock Approximate discrete-time schemes for statistics of diffusion processes.
\newblock {\em Statistics: A Journal of Theoretical and Applied Statistics}, 20(4):547--557, 1989.

\bibitem{frank1993statistical}
L.~E. Frank and J.~H. Friedman.
\newblock A statistical view of some chemometrics regression tools.
\newblock {\em Technometrics}, 35(2):109--135, 1993.

\bibitem{fujimori2019dantzig}
K.~Fujimori.
\newblock The dantzig selector for a linear model of diffusion processes.
\newblock {\em Statistical Inference for Stochastic Processes}, 22:475--498, 2019.

\bibitem{gaiffas2019sparse}
S.~Ga{\"\i}ffas and G.~Matulewicz.
\newblock Sparse inference of the drift of a high-dimensional ornstein--uhlenbeck process.
\newblock {\em Journal of Multivariate Analysis}, 169:1--20, 2019.

\bibitem{10.1214/18-EJS1436}
A.~D. Gregorio and S.~M. Iacus.
\newblock {On penalized estimation for dynamical systems with small noise}.
\newblock {\em Electronic Journal of Statistics}, 12(1):1614 -- 1630, 2018.

\bibitem{degregorio2024pathwiseoptimizationbridgetypeestimators}
A.~D. Gregorio and F.~Iafrate.
\newblock Pathwise optimization for bridge-type estimators and its applications, https://arxiv.org/abs/2412.04047.
\newblock 2024.

\bibitem{hastie2015statistical}
T.~Hastie, R.~Tibshirani, and M.~Wainwright.
\newblock Statistical learning with sparsity.
\newblock {\em Monographs on statistics and applied probability}, 143(143):8, 2015.

\bibitem{sdelearn}
F.~Iafrate.
\newblock {SDELearn: a Python package for SDE Modeling}.
\newblock \url{https://github.com/fiafrate/sdelearn}, 2024.
\newblock Version 0.1.2.

\bibitem{enet-sde}
F.~Iafrate and D.~Frisardi.
\newblock Elastic net estimation for ergodic diffusion processes.
\newblock \url{https://github.com/fiafrate/enet-sde}, 2024.

\bibitem{kamatani2015hybrid}
K.~Kamatani and M.~Uchida.
\newblock Hybrid multi-step estimators for stochastic differential equations based on sampled data.
\newblock {\em Statistical Inference for Stochastic Processes}, 18(2):177--204, 2015.

\bibitem{kess}
M.~Kessler.
\newblock Estimation of an ergodic diffusion from discrete observations.
\newblock {\em Scandinavian Journal of Statistics}, 24(2):211--229, 1997.

\bibitem{kinoshita2019penalized}
Y.~Kinoshita and N.~Yoshida.
\newblock Penalized quasi likelihood estimation for variable selection.
\newblock {\em arXiv preprint arXiv:1910.12871}, 2019.

\bibitem{kloeden1992}
P.~E. Kloeden and E.~Platen.
\newblock {\em Numerical Solution of Stochastic Differential Equations}.
\newblock Springer Berlin Heidelberg, 1992.

\bibitem{masuda2017moment}
H.~Masuda and Y.~Shimizu.
\newblock Moment convergence in regularized estimation under multiple and mixed-rates asymptotics.
\newblock {\em Mathematical Methods of Statistics}, 26:81--110, 2017.

\bibitem{pardoux2001poisson}
{\'E}.~Pardoux and Y.~Veretennikov.
\newblock On the poisson equation and diffusion approximation. i.
\newblock {\em The Annals of Probability}, 29(3):1061--1085, 2001.

\bibitem{sorensen2024efficient}
M.~S{\o}rensen.
\newblock Efficient estimation for ergodic diffusion processes sampled at high frequency.
\newblock {\em arXiv preprint arXiv:2401.04689}, 2024.

\bibitem{suzuki2020penalized}
T.~Suzuki and N.~Yoshida.
\newblock Penalized least squares approximation methods and their applications to stochastic processes.
\newblock {\em Japanese Journal of Statistics and Data Science}, 3(2):513--541, 2020.

\bibitem{tibshirani1996regression}
R.~Tibshirani.
\newblock Regression selection and shrinkage via the lasso.
\newblock {\em Journal of the Royal Statistical Society Series B}, 58(1):267--288, 1996.

\bibitem{uchida2012adaptive}
M.~Uchida and N.~Yoshida.
\newblock Adaptive estimation of an ergodic diffusion process based on sampled data.
\newblock {\em Stochastic Processes and their Applications}, 122(8):2885--2924, 2012.

\bibitem{wang1}
H.~Wang and C.~Leng.
\newblock Unified lasso estimation by least squares approximation.
\newblock {\em Journal of the American Statistical Association}, 102(479):1039--1048, 2007.

\bibitem{yoshida2011polynomial}
N.~Yoshida.
\newblock Polynomial type large deviation inequalities and quasi-likelihood analysis for stochastic differential equations.
\newblock {\em Annals of the Institute of Statistical Mathematics}, 63(3):431--479, 2011.

\bibitem{yoshida2022nonlin}
N.~Yoshida.
\newblock Quasi-likelihood analysis for nonlinear stochastic processes.
\newblock {\em Econometrics and Statistics}, 2022.

\bibitem{zou2006adaptive}
H.~Zou.
\newblock The adaptive lasso and its oracle properties.
\newblock {\em Journal of the American statistical association}, 101(476):1418--1429, 2006.

\bibitem{zou2005regularization}
H.~Zou and T.~Hastie.
\newblock Regularization and variable selection via the elastic net.
\newblock {\em Journal of the Royal Statistical Society Series B: Statistical Methodology}, 67(2):301--320, 2005.

\bibitem{zou2009adaptive}
H.~Zou and H.~H. Zhang.
\newblock On the adaptive elastic-net with a diverging number of parameters.
\newblock {\em Annals of statistics}, 37(4):1733, 2009.

\end{thebibliography}

\end{document}